\documentclass[a4paper]{amsart}

\usepackage{amsmath}
\usepackage{amsfonts}
\usepackage{amssymb}
\usepackage{amsrefs}
\usepackage{graphicx}
\usepackage{tikz}

\usetikzlibrary{matrix,arrows}
\newtheorem{theorem}{Theorem}[section]

\theoremstyle{plain}
\newtheorem{corollary}[theorem]{Corollary}
\newtheorem{lemma}[theorem]{Lemma}
\newtheorem{proposition}[theorem]{Proposition}

\newtheorem*{theoremA}{Theorem A}
\newtheorem*{theoremB}{Theorem B}
\newtheorem*{lemma*}{Lemma}
\newtheorem*{proposition*}{Proposition}

\theoremstyle{definition}
\newtheorem{definition}[theorem]{Definition}
\newtheorem{example}[theorem]{Example}

\newcommand{\CC}{\mathbb{C}}

\newcommand{\KK}{\mathbb{K}}

\newcommand{\RR}{\mathbb{R}}

\newcommand{\calB}{\mathcal{B}}

\newcommand{\calH}{\mathcal{H}}
\newcommand{\calK}{\mathcal{K}}

\newcommand \lcur{[\![}
\newcommand \rcur{]\!]}

\newcommand{\ii}{\sqrt{-1}}

\DeclareMathOperator{\Val}{Val}
\DeclareMathOperator{\vol}{vol}
\DeclareMathOperator{\id}{id}
\DeclareMathOperator{\supp}{supp}
\DeclareMathOperator{\Area}{Area}
\DeclareMathOperator{\glob}{glob}

\DeclareMathOperator{\Sym}{Sym}
\DeclareMathOperator{\Hom}{Hom}
\DeclareMathOperator{\PD}{pd}
\DeclareMathOperator{\contr}{contr}
\usepackage{fullpage}
\usepackage{hyperref}
\begin{document}
\author{Thomas Wannerer}
\address
	{
		\begin{flushleft}	
			Goethe-Universit\"at Frankfurt\\
			Institut f\"ur Mathematik\\
			Robert-Mayer-Str.\ 10\\
			60325 Frankfurt am Main, Germany \\
		\end{flushleft}
	}	
	
	\email{wannerer@mathematik.uni-frankfurt.de}

\title{Integral geometry of unitary area measures}
\thanks{Research supported by DFG grant BE 2484/5-1}

\begin{abstract} The existence of kinematic formulas for area measures with respect to any connected, closed subgroup of the orthogonal group acting transitively on the unit sphere is established. 
In particular, the kinematic operator for area measures is shown to have the structure of a co-product.
 In the case of the unitary group the algebra associated to this co-product is described explicitly in terms of generators and relations. As a consequence, a simple algorithm that yields explicit kinematic formulas
 for unitary area measures is obtained.  
\end{abstract}

\maketitle

\section{Introduction}
The answer to the question \textit{``What is the expected volume of the Minkowski sum of two convex bodies $K,L\subset \RR^n$ if $L$ is rotated randomly?''} is given by  the additive principal kinematic formula
\begin{equation}
 \label{eq:principal}
\int_{SO(n)} \vol_n(K +gL) \; dg = \sum_{i+j=n} \binom{n}{i}^{-1} \frac{\omega_i\omega_j}{\omega_n} \mu_i(K) \mu_j(L),
\end{equation}
first proved by Blaschke \cite{blaschke55} for $n=2,3$ and in general by Chern \cite{chern52}. Here $SO(n)$ denotes the rotation group equipped with the Haar probability measure,
 $\omega_i$ is the volume of the $i$-dimensional euclidean unit ball and $\mu_i(K)$ is the $i$th intrinsic volume of $K$, a suitably normalized $(n-i-1)$-th order mean curvature integral. The principal kinematic formula 
\eqref{eq:principal} plays a central role in classical integral geometry, since it encompasses many results in euclidean
integral geometry as special or limiting cases (see, e.g., \cites{klain_rota97,santalo04,schneider_weil08} for the history and applications of integral geometry). In this paper we establish a hermitian,
 local version of \eqref{eq:principal}. Here `local' means that the intrinsic volumes may be replaced
by local curvature integrals.

Over the last ten years there have been tremendous advances in integral geometry that completely reshaped the subject. The foundation for this progress was the discovery 
of various algebraic structures on the space of valuations by Alesker \cites{alesker01, alesker03,alesker04}. Here a valuation is a function $\phi\colon \calK(\RR^n)\to \RR$ on the space of convex bodies satisfying 
the additivity property 
\begin{equation}\label{eq:def_val}\phi(K\cup L)= \phi(K) + \phi(L) -\phi(K\cap L)\end{equation}
whenever $K\cup L $ is convex. The theory of valuations, which turned out to be the key to modern integral geometry, is a very rich one and dates back to Dehn's solution of Hilbert's 3rd problem 
(see \cites{abardia_etal12,alesker99b,haberl12, haberl_parapatits13,klain99,ludwig03,ludwig_reitzner10, parapatits_wannerer13, schuster10} and the references therein).

Building on the work of Alesker, it was first realized by Fu \cite{fu06} and then put in more precise terms by Bernig and Fu \cite{bernig_fu06}, that the classical kinematic operators are  co-products
(which is reflected by the bilinear structure of the principal kinematic formula \eqref{eq:principal}) and are in this sense dual to the only recently discovered products on valuations. 
 This crucial discovery provided not 
only an explanation for the algebraic-combinatorial behavior of the numeric constants appearing in kinematic formulas, but also opened the door to determining kinematic formulas explicitly in
 cases different from the euclidean. For example,
the problem of evaluating the integral in \eqref{eq:principal}
 with  $\RR^n$ replaced by $\CC^n$ and
$SO(n)$ by the unitary group $U(n)$ seemed for many years out of reach, but has recently been solved by Bernig and Fu in the landmark paper \cite{bernig_fu11}.
 
In this article we generalize the results of Bernig and Fu \cite{bernig_fu11} and determine the local additive kinematic formulas for the unitary group completely. Using a different approach, 
the problem of localizing the intersectional 
kinematic formulas has only recently been solved by Bernig, Fu, and Solanes \cite{bernig_etal13}.
 To explain what we mean by `local' let us first consider
the classical euclidean case. It is a well-known fact from the geometry of convex bodies that to each convex body $K\subset \RR^n$ one can associate a measure $S_{n-1}(K,\; \cdot\;)$ on the unit sphere $S^{n-1}$,
 called the area measure of $K$, such that 
\begin{equation}
 \label{eq:vol_areameasure}
\vol_n(K)= \frac{1}{n} \int_{ S^{n-1}} \left\langle u,x\right\rangle \; dS_{n-1}(K,u),
\end{equation}
where $x$ is a boundary point of $K$ with outer unit normal $u$ and $\left\langle u,x\right\rangle$ denotes the euclidean inner product. If the boundary of $K$ is $C^2$ and has all principal curvatures positive, 
then the measure $S_{n-1}(K,\; \cdot \;)$ is absolutely continuous with density 
equal to the reciprocal of the Gauss curvature at $x$. Schneider \cite{schneider75} proved that for all convex bodies $K,L\subset \RR^n$ and Borel subsets $U,V\subset S^{n-1}$ of the unit sphere 
\begin{equation}\label{eq:principal_local}\int_{SO(n)} S_{n-1}(K+ gL, U\cap g V) \; dg = \frac{1}{n\omega_n}  \sum_{i+j=n-1}  \binom{n-1}{i} S_i(K,U) S_j(L,V),\end{equation}
where $S_{i}(K,\;\cdot\;)$ denotes the $i$th area measure of $K$. Using relation \eqref{eq:vol_areameasure} and similar relations between $\mu_{i}(K)$ and $S_{i-1}(K,\; \cdot\;)$, 
the global principal kinematic formula \eqref{eq:principal} may be deduced from its local version \eqref{eq:principal_local}.

Let us describe now the results of the present article. It is far from obvious that one may replace $\RR^n$ by $\CC^n=\RR^{2n}$ and $SO(n)$ by $U(n)$ and still obtain a
bilinear splitting of the integral in \eqref{eq:principal_local}. In Section~\ref{sec:existence} we use a method developed by Fu \cite{fu90} based on fiber integration to prove that the
 special orthogonal group may be replaced by any connected,  closed subgroup $G\subset SO(n)$
acting transitively on the unit sphere. In particular, we may replace $SO(n)$ by $U(n)$.

Let us denote by $\Val^G$ the space of $G$-invariant,  translation-invariant, continuous valuations \cite{bernig_fu06} and by $\Area^G$ the space of $G$-invariant area measures \cite{wannerer13}. 
As indicated above, in the case of global kinematic formulas, 
the kinematic operator is a co-product which corresponds (after an identification of $\Val^{G}$ with its dual space via Poincar\'e duality \cite{alesker04}) to a product on $\Val^G$. This statement is known as the 
fundamental theorem of algebraic integral geometry (ftaig), see \cite{bernig_fu06}.
 If we want to use this approach to find explicit 
local kinematic formulas, we are faced with two problems: First, no product on $\Area^G$ is known and second, it is not clear how to identify $\Area^{G}$ with its dual space. 
However, as was shown by the author in \cite{wannerer13}, $\Area^G$ possesses the structure of a module over $\Val^G$. This property and the techniques of \cite{wannerer13} will play an important role in the present article.
Lacking a canonical identification of $\Area^{G}$ with its dual space, we will work directly with the product on $\Area^{G*}$ induced by the kinematic operator. This approach is new and will turn out to be
 surprisingly convenient.

We denote by $\RR[ s, t,  v]$ the polynomial algebra generated by the variables $ s$, $t$, and $ v$. 
The main result of this paper is an explicit description of the algebra $\Area^{U(n)*}$.

\begin{theoremA}
 The algebra $\Area^{U(n)*}$ is isomorphic to $\RR[s,  t, v] /I_n$, where the ideal $I_n$ is generated by
\begin{equation}\label{eq:algebra_relations}f_{n+1}(s, t), \qquad f_{n+2}( s, t),\qquad  p_n( s, t)-q_{n-1}( s, t) v,\qquad  p_n( s, t) v,\end{equation}
and
$$v^2.$$
The polynomials $f_k$, $p_k$, and $q_k$ are given by the Taylor series expansions
\begin{align*}
 \log(1+  t x +  s x^2) &= \sum_{k=0}^\infty f_k( s, t) x^k,\\
 \frac{1}{1+ t x +  s x^2} &= \sum_{k=0}^\infty p_k( s, t) x^k, \ \text{and } \\
 -\frac{1}{(1+  t x + s x^2)^2} &= \sum_{k=0}^\infty q_k( s,  t) x^k.
\end{align*}
\end{theoremA}
In Section~\ref{sec:explicit} we show how Theorem~A can be used to achieve the ultimate goal of obtaining explicit kinematic formulas for unitary area measures. 
We give examples and provide a simple algorithm which can be implemented in any computer algebra package.

We remark that the polynomials $f_k$ have appeared already in Fu's description of the unitary valuation algebra \cite{fu06} and so have the polynomials $p_k$ and $q_k$ in the description  of 
the module of unitary area measures by the author \cite{wannerer13}. The isomorphism in Theorem~A is given explicitly by
\begin{align*} t \quad & \longmapsto \quad \bar t:=\frac{2 \omega_{2n-2}}{\omega_{2n-1}}(B_{1,0}^* + \Gamma_{1,0}^*),\\
		s \quad & \longmapsto \quad \bar s:=\frac{n}{\pi} \Gamma^*_{2,1}, \text{ and}\\
		v \quad & \longmapsto \quad \bar v:=\frac{2\omega_{2n-2}}{\omega_{2n-1}} B_{1,0}^*.
\end{align*}
Here $B_{k,p}^*$, $\Gamma_{k,q}^*$ are elements of the basis dual to the $B_{k,q}$-$\Gamma_{k,q}$ basis for $\Area^{U(n)}$ used in \cite{wannerer13}.

Once the general set-up has been developed, the main difficulty in proving Theorem~A lies in proving $\bar v^2=0$; in fact, using the techniques of \cite{wannerer13}, the other relations \eqref{eq:algebra_relations} can be deduced from the explicit description of
the module $\Area^{U(n)}$, which has been obtained by the author in \cite{wannerer13}. 
The proof of $\bar v^2=0$, however, requires once again new ideas and tools. 

The main new idea is to relate the kinematic operator for area measures with a new kinematic operator for tensor valuations, i.e.\ valuations in the sense of \eqref{eq:def_val}, but with values in 
the symmetric algebra on $\RR^n$. Although tensor valuations compatible with the special orthogonal group have been studied
 by several authors \cites{alesker99a,alesker_etal11,hug_etal07,mcmullen97} and a family of kinematic formulas  has been established \cite{hug_etal08}, the integral geometry of tensor valuations is for the most part
still unexplored. In recent years, tensor valuations have received increased attention in applied sciences \cites{beisbart_etal02, beisbart_etal06, schroeder_etal10}.

Let $\Val^{r,G}$ denote the space of translation-invariant, continuous, $G$-covariant, rank $r$  tensor valuations. Since already in the case $G=SO(n)$ 
a full investigation of these valuations and their integral geometry would be out of the scope of this article,
we confine ourselves to introduce only those concepts which are necessary to prove Theorem A. In Section~\ref{sec:tensor}, we introduce for nonnegative integers $r_1,r_2\geq 0$ a kinematic operator
$a^{r_1,r_2}\colon \Val^{r_1+r_2,G}\rightarrow \Val^{r_1,G}\otimes \Val^{r_2,G}$ by
$$a^{r_1,r_2}(\Phi)(K,L) = \int_G (\operatorname{id}^{\otimes r_1} \otimes g^{\otimes r_2}) \Phi(K+ g^{-1}L)\; dg$$
whenever $K,L\subset \RR^n$ are convex bodies and define for $r\geq 0$ the $r$th order moment map $M^r\colon \Area^G\rightarrow \Val^{r,G}$ by
$$M^r(\Phi)(K) = \int_{S^{n-1}} u^{ r} \; d\Phi(K,u).$$ 
We then prove that
\begin{equation}\label{eq:intro_momentkin}
 a^{r_1,r_2} \circ M^{r_1+r_2}(\Phi) = (M^{r_1}\otimes M^{r_2}) \circ A (\Phi),\qquad \Phi\in \Val^{r_1+r_2,G}
\end{equation}
where $A\colon \Area^G\rightarrow \Area^G\otimes \Area^G$ denotes the kinematic operator for area measures. Moreover, we extend the Bernig-Fu convolution and Alesker's Poincar\'e duality for scalar valuations to 
tensor valuations. In particular, we show that a version of the fundamental theorem of algebraic integral geometry (ftaig) holds for tensor valuations:

\begin{theoremB}
 Let $r_1,r_2$ be nonnegative integers and denote by $c_G\colon \Val^{r_1,G}\otimes \Val^{r_2,G}\rightarrow \Val^{r_1+r_2, G}$ and $\widehat{\PD}\colon \Val^{r,G}\rightarrow (\Val^{r,G})^*$ the convolution product and
 Poincar\'e duality. Then
$$a^{r_1,r_2} = (\widehat{\PD}^{-1}\otimes \widehat{\PD}^{-1}) \circ c^*_G \circ \widehat{\PD}.$$
\end{theoremB}

After these general results, we show that the second order moment map $M^2$ is injective in the case $G=U(n)$. Hence \eqref{eq:intro_momentkin} with $r_1=r_2=2$ can be used to obtain --- at least in principle --- the
 kinematic formula for area measures from the kinematic formulas for tensor valuations. In a final step, we use Theorem~B to prove $\bar v^2=0$.

\subsection*{Acknowledgements} I would like to thank the University of Georgia for their hospitality at an early stage of this work.
  I am grateful to Andreas Bernig and Joe Fu for illuminating discussions and their encouragement.

\section{Existence of kinematic formulas for area measures} \label{sec:existence}

The goal of this section is to  prove that if we replace the special orthogonal group by any connected, closed subgroup $G\subset SO(n)$ acting transitively on the unit sphere, we still
obtain a bilinear splitting of the integral in \eqref{eq:principal_local}. To make this statement precise, we need to introduce some notation.

In the following $G\subset SO(n)$ denotes a connected, closed subgroup acting transitively on the unit sphere. We will always assume that $n\geq 2$. We denote by $S\RR^n $ the sphere bundle of $\RR^n$ 
and write $\pi_1\colon S\RR^n\rightarrow \RR^n$ and $\pi_2\colon S\RR^n\rightarrow S^{n-1}$ for the natural projections. The group $\overline{G}=G\ltimes \RR^n$ generated by 
$G$ and translations of $\RR^n$ acts on $\RR^n$ by isometries and hence there is a canonical action of $\overline{G}$ on the sphere bundle. The space of differential forms on $S\RR^n$ invariant under this action is denoted by $\Omega(S\RR^n)^G$. Since $G$ acts transitively
on the unit sphere, $\Omega(S\RR^n)^G$ is finite-dimensional. If $\phi:S^{n-1}\rightarrow \RR$ is a function on the unit sphere, we denote by $(g \phi)(u)=\phi(g^{-1}u)$ left translation by $g\in G$.

Let $K\subset \RR^n$ be a convex body, i.e.\ a nonempty, compact, convex subset. The normal cycle $N(K)$ of $K$ consists of those pairs $(x,u)\in S\RR^n$ such that $x$ is a 
boundary point of $K$ and $u$ is an outer unit normal of $K$ at $x$. Since $N(K)$ 
is a countably $(n-1)$-rectifiable set and carries a natural orientation (see \cites{fu90, fu94,alesker_fu08}), we can integrate $(n-1)$-forms over it, 
$$N(K)(\omega):=\int_{N(K)} \omega,\qquad \omega\in \Omega^{n-1}(S\RR^n).$$
For example, it is not difficult to see that there exists a form $\kappa_{n-1}\in \Omega^{n-1}(S\RR^n)^{SO(n)}$ such that 
$$ \int_{S^{n-1}} \phi(u)\; dS_{n-1}(K,u) = N(K)(\pi^*_2 \phi \cdot  \kappa_{n-1})$$
for every bounded Baire function $\phi$.

\begin{theorem}\label{thm:existence}
Let  $\beta_1,\ldots, \beta_m$ be a basis of $\Omega^{n-1}(S\RR^n)^G$. 
If $\omega\in \Omega^{n-1}(S\RR^n)^G$, then there exist constants $c_{ij}$ such that
$$\int_G N(K+gL)( \pi_2^*(\phi \; g\psi) \cdot  \omega)  \; dg = \sum_{i,j}^{m} c_{ij} N(K)( \pi_2^*\phi \cdot \beta_i)\;   N(L)( \pi_2^*\psi \cdot \beta_j)$$
for all bounded Baire functions $\phi, \psi$ on $S^{n-1}$  and convex bodies  $K, L\subset \RR^n$.

\end{theorem}

In \cite{fu90} Fu proved the existence of intersectional kinematic formulas in a very general setting. In this 
section we adapt Fu's method to prove Theorem~\ref{thm:existence}.

\subsection{Fiber bundles and fiber integration} Before we prove the existence of general additive kinematic formulas, we first need to recall the definition and main properties of fiber integration, cf.\ 
Chapter~VII of \cite{greub_etal72} and \cite{fu90}. 

Let $\calB=(E,\pi, M, F)$ be a smooth fiber bundle with total space $E$,
base space $M$, projection $\pi\colon E\rightarrow M$ and fiber $F$. Recall that the bundle $\calB$ is orientable if and only if the fiber $F$ is orientable and there exists an open cover 
$\{U_i\}$ of $M$ and local trivializations $\phi_i \colon \pi^{-1}(U_i) \rightarrow U_i \times F$ such that the transition maps 
$c_{ij}(x)\colon F\rightarrow F$, $\phi_i \circ \phi_j^{-1}(x,y)=(x,c_{ij}(x)y)$, are orientation preserving diffeomorphisms. A bundle $\calB$ is  oriented by an orientation of $F$ together with an open cover of the 
base corresponding to a collection of local trivializations with orientation preserving transition maps.
A differential form $\omega$ on $E$ is said to have fiber-compact support if for every compact $K\subset M$ the intersection $\pi^{-1}(K)\cap \supp \omega$ is compact. The space of fiber-compact differential forms
is denoted by $\Omega_F(E)$.

If $(E,\pi,M,F)$ is an oriented bundle with $\dim M= n$ and $\dim F=r$, then there exists a canonical linear map 
$\pi_*\colon  \Omega_F(E) \rightarrow \Omega(M)$ called fiber integration. It is defined as follows: 

Let $\omega$ be a fiber-compact form on $E$ and suppose $\phi:\pi^{-1}(U)\rightarrow U\times F$ is a 
local trivialization compatible with the orientation of the bundle. Shrinking $U$ if necessary, we may assume that $x_1,\ldots,x_n$ are coordinates for $U$. Then $(\phi^{-1})^*\omega$ can be written uniquely
as 
$$(\phi^{-1})^*\omega(x,y) = \sum_{k=0}^n\sum_{i_1<\ldots < i_k} dx_{i_1}\wedge \cdots \wedge dx_{i_k} \wedge \omega^{(i_1,\ldots,i_k)}(x,y),\qquad (x,y)\in U\times F,$$
where $\partial x_j\lrcorner \omega^{(i_1,\ldots,i_k)}(x,y)=0$ for every coordinate vector field $\partial x_j$. Then $y\mapsto \omega^{(i_1,\ldots,i_q)}(x,y)$ is an element of $\Omega_c(F)$
and the integral of $\omega$ over the fibers is defined by
$$\pi_*\omega(x):=\sum_{k=0}^n\sum_{i_1<\ldots < i_k} dx_{i_1}\wedge \cdots \wedge dx_{i_k} \int_F \omega^{(i_1,\ldots,i_k)}_x$$
with the convention $\int_F \xi=0$ for forms $\xi$ of degree less than $r$.

If $M$ is orientable and $E$ is equipped with the local product orientation (see \cite{greub_etal72}*{p. 288}), then
$$\int_M \alpha \wedge \pi_* \omega= \int_E \pi^* \alpha \wedge \omega$$
for every $\omega\in \Omega_F^k(E)$ and $\alpha\in \Omega_c^{n+r-k}(M)$. In particular, if $N\subset M$  is an oriented, compact submanifold with $\dim N= q$ and $\pi^{-1}(N)$ is equipped with the local product orientation, then the following version of Fubini's theorem holds: 
If $\omega\in \Omega_F^{q+r}(E)$, then
\begin{equation}\label{eq:fiber_fubini}
 \int_N  \pi_* \omega  =\int_{\pi^{-1}(N)}  \omega.
\end{equation}

Let $(E,\pi, M, F)$ and $(E', \pi',M',F)$ be oriented bundles with the same fiber $F$ and let $\bar f\colon E'\rightarrow E$ be a bundle map covering the smooth map $f\colon M'\rightarrow M$.
If there exists an open cover $\{U\}$ of $M$ together with local trivializations $\phi\colon \pi^{-1}(U)\rightarrow U\times F$ and $\phi'\colon \pi^{-1}(f^{-1}(U))\rightarrow f^{-1}(U)\times F$ compatible with the
orientations of the bundles such that
$$\phi^{-1}\circ \bar f\circ \phi' = f\times \operatorname{id}_F,$$
then
\begin{equation}
 \label{eq:bundlemap}
f^*\circ \pi_* = \pi_*' \circ \bar f^*,
\end{equation}
see \cite{fu90}.

\subsection{A slicing formula}

It will be helpful to restate the general slicing formula for currents \cite{federer69}*{4.3.2} in the special setting in which we are going to apply it. 

Suppose $f\colon X\to Y$ is a surjective, smooth map between compact, smooth manifolds with $m=\dim X$ and $n=\dim Y$. Suppose that $X$ is oriented by a
 smooth $m$-vector field $\xi$ and that $Y$ is oriented by a smooth $n$-form $dy$. By Sard's theorem, $f^{-1}(y)$ is a smooth submanifold for almost every $y\in Y$ and is 
oriented by the smooth $(m-n)$-vector field
\begin{equation}
 \label{eq:slice_orientation}
\zeta = \xi\llcorner f^*dy
\end{equation}
on $X$. If we define the measure $\mu$ by $\lcur Y \rcur \llcorner dy$, then by \cite{federer69}*{4.3.2, 4.3.8} 
\begin{equation}
 \label{eq:slicing}
\int_Y \Phi(y) \left( \int_{f^{-1}(y) } \omega \right)  \; d\mu(y) = \int_X f^*(\Phi \wedge dy)\wedge \omega
\end{equation}
for every bounded Baire function $\Phi\colon Y\to \RR$ and every $(m-n)$-form $\omega$ on $X$.

\subsection{Proof of Theorem~\ref{thm:existence}}

Let us fix a point $o\in S^{n-1}$ and let $G_o\subset G$ be the isotropy group of $o$. Note that $G_o$ is always connected. Indeed, if $n=2$, then $G_0=\{e\}$ and if $n\geq 3$, then $S^{n-1}$ is simply connected.

We put
$$E= \{ (g,\xi,\eta, \zeta)\in G\times (S\RR^n )^3: \pi_1(\zeta)= \pi_1(\xi) + g\pi_1(\eta), \ \pi_2(\zeta)= \pi_2(\xi)= g\pi_2(\eta)\}$$
and define a projection $p\colon  E\rightarrow S\RR^n \times S\RR^n$ by $p(g,\xi,\eta, \zeta)=(\xi, \eta)$. We claim that $p\colon E\rightarrow S\RR^n \times S\RR^n$ is an orientable fiber bundle with fiber $G_o$.
In fact, suppose $U,V\subset S\RR^n$ are open sets and $\phi\colon U \rightarrow G$, $\psi\colon V \rightarrow G$ are smooth maps satisfying $\phi(\xi)o= \pi_2(\xi)$, $\xi\in U$, and $\psi(\eta) o = \pi_2(\eta)$,  $\eta\in V$. Then
\begin{align*} \Phi\colon &p^{-1}(U\times V) \longrightarrow (U\times V) \times G_o \\
					      &(\xi,\eta, \zeta, g)\mapsto (\xi,\eta,\phi(\xi)^{-1} g \psi(\eta))
\end{align*}
is a local trivialization for $p$. 
If $\Phi'\colon p^{-1}(U'\times V') \longrightarrow (U'\times V') \times G_o$ is another local trivialization constructed in the same way, then the transition map $c$ is given by
\begin{align*}
 c \colon & (U\times V) \cap (U'\times V') \rightarrow \operatorname{Diff}(G_o)\\
			& (\xi, \eta)\mapsto \left[ g \mapsto (\phi')(\xi)^{-1} \phi(\xi) \cdot g\cdot  \psi(\eta)^{-1} \psi'(\eta)\right].
\end{align*}
Note that both $(\phi')(\xi)^{-1} \phi(\xi)$ and $ \psi(\eta)^{-1} \psi'(\eta)$ are elements of $G_o$. Since  $G_o$ is compact and connected, left and right translations are orientation preserving.
Hence the transition map consists of orientation preserving diffeomorphisms. We conclude that our bundle $(E, p, S\RR^n \times S\RR^n, G_o)$ is orientable.

We define another projection $q\colon E\rightarrow G\times S\RR^n$ by $p(g,\xi,\eta, \zeta)=(g,\zeta)$. Observe that $q\colon E\rightarrow G\times S\RR^n$ is isomorphic 
to the trivial bundle $G\times S\RR^n\times \RR^n$.

For every element $g\in \overline G$ we define $g_0\in G$ by $g_0(x):=g(x)-g(0)$. We let the group $\overline{G}\times \overline{G}$ act on the total 
space $E$ by 
$$(h,k)\cdot(g,\xi,  \eta,  \zeta):= (h_0gk_0^{-1}, h \xi, k \eta, h \zeta),\qquad (h,k)\in \overline{G}\times \overline{G},$$
and on the base spaces $S\RR^n\times S\RR^n$  and $G\times S\RR^n$ 
by 
$$(h,k)\cdot (\xi,\eta):= (h\xi,k\eta)\qquad \text{and} \qquad (h,k)\cdot (g,\zeta):=(h_0gk_0^{-1}, h \zeta).$$ 
Clearly, the projections $p$ and $q$ commute with these group actions.

\begin{lemma}\label{lem:invariance}
Suppose $\omega\in\Omega(E)$ is  $\overline{G}\times \overline{G}$-invariant. Then the fiber integral  $p_*\omega$ is a $\overline{G}\times \overline{G}$-invariant form on $S\RR^n\times S \RR^n$.
\end{lemma}
\begin{proof}
 Fix $(h,k)\in \overline{G}\times \overline{G}$ and $(\xi_0,\eta_0)\in S\RR^n\times S\RR^n$. Choose open neighborhoods $U$, $V$ of $\xi_0$ and $\eta_0$ such that there exist smooth maps
 $\phi\colon U \rightarrow G$ and  $\psi\colon V \rightarrow G$ with
$\phi(\xi)o= \pi_2(\xi)$, and $\phi(\eta) o = \pi_2(\eta)$. Put $W=U\times V$ and define a local trivialization 
\begin{align*}
  \Phi\colon &p^{-1}(W) \longrightarrow W \times G_o \\
					      &(\xi,\eta, \zeta, g)\mapsto (\xi,\eta,\phi(\xi)^{-1} g \psi(\eta)).
\end{align*}
Define also maps $f\colon S\RR^n\times S\RR^n\rightarrow S\RR^n\times S\RR^n$ and $\bar f\colon E\rightarrow E$ by
$$f(\xi,\eta) = (h,k)\cdot (\xi, \eta) \quad \text{and} \quad \bar f(g,\xi,\eta, \zeta) =(h,k)\cdot  (g,\xi, \eta, \zeta)$$
Observe that
\begin{align*}
  \Phi'\colon &p^{-1}(f(W)) \longrightarrow f(W) \times G_o \\
					      &(\xi,\eta, \zeta, g)\mapsto (\xi,\eta,\phi(\xi)^{-1}h_0^{-1} g k_0\psi(\eta)).
\end{align*}
is also a local trivialization and that
$$\Phi'^{-1} \circ \bar f\circ \Phi = f\times \operatorname{id}_{G_o}.$$
Thus we can apply \eqref{eq:bundlemap} and obtain $f^*\circ p_* = p_* \circ \bar f^*$.
This shows that $p_*\omega$ is $\overline{G}\times \overline{G}$-invariant.
\end{proof}

Finally, we need to recall two facts concerning the normal cycles of convex bodies (see, e.g., the proof of Lemma~2.1.3 in \cite{alesker_fu08}).

\begin{lemma}\label{lem_massbound}
 Let $\{K_\alpha \}$ be a collection of convex bodies. If there exists some ball containing all $K_\alpha$, then there exists a constant $C$ such that $$\sup_\alpha \mathbf{M}(N(K_\alpha))\leq C,$$ where $\mathbf{M}(T)$
denotes the mass of a current $T$, see \cite{federer69}*{p. 358}. 
\end{lemma}

\begin{lemma}\label{lem_weakconvergence}
 Let $\omega\in \Omega^{n-1}(S\RR^n)$ and let $f: S^{n-1}\rightarrow \RR$ be a continuous function. If $K_i\rightarrow K$ is a convergent sequence of convex bodies in $\RR^n$, then

$$N(K_i)(\pi_2^* f\cdot \omega)\rightarrow N(K)(\pi_2^*f\cdot \omega).$$
\end{lemma}

\begin{proof}[Conclusion of the proof of Theorem~\ref{thm:existence}]

We will first prove Theorem~\ref{thm:existence} under the additional assumption that $\phi, \psi$ are smooth and that $K, L$ have smooth boundaries and all principal curvatures are positive. The general case
will then follow by approximation.

We define $q_1\colon E\rightarrow G$ by $q_1(g,\xi,\eta,\zeta)=g$ and $q_2\colon E\rightarrow S\RR^n$ by $q_2(g,\xi,\eta,\zeta)= \zeta$. If we put $C= p^{-1}(N(K)\times N(L))$, then
$$ q_2(q_1^{-1}(g)\cap C) = N(K+g L).$$ 
In fact, this was the motivation for the definition of the bundle $E$. 

Let $dg$ be a bi-invariant volume form on $G$ such that $\int_G dg=1$. Since $q_1\colon C\rightarrow G$ is a smooth surjective map, Sard's theorem implies that $q_1(g)\cap C$ is a smooth submanifold for a.e.\ $g\in G$. 
Moreover, since the boundary of $K$ contains no segments,
 $q_2 \colon q_1^{-1}(g)\cap C \rightarrow N(K+gL)$ is a diffeomorphism for a.e.\ $g\in G$. Since the bundle $p\colon C\to N(K)\times N(L)$ is orientable and the normal cycles carry a natural orientation,
 we can equip $C$ with the local product orientation. By \eqref{eq:slice_orientation}, this induces an orientation 
on almost every slice $q_1^{-1}(g)\cap C$. We choose the orientation for the bundle such that $q_2\colon q_1^{-1}(g)\cap C \to N(K+gL)$ is orientation preserving for a.e.\ $g\in G$. 
Hence 
$$N(K+gL)(\omega)  = \int_{q_1^{-1}(g)\cap C} q_2^* \omega \qquad \text{for a.e. } g\in G$$
and for every $(n-1)$-form $\omega$ on $S\RR^n$.

We define $p_1\colon E\rightarrow S\RR^n$ by $p_1(g,\xi,\eta,\zeta)= \xi$ and $p_2\colon E\rightarrow S\RR^n$ by $p_2(g,\xi,\eta,\zeta)=\eta$. 
Since $q_2^*\pi_2^* \phi = p^*_1 \pi_2^*\phi$ and $q_2^*\pi_2^*(g \psi) = p^*_2 \pi_2^*\psi$,
we obtain
$$ q^*_2(\pi_2^*(\phi \; g\psi) \cdot \omega) = p^*_1 \pi_2^*\phi \cdot p^*_2 \pi_2^*\psi\cdot q^*_2\omega=:\omega'.$$

  Using the slicing formula \eqref{eq:slicing} and \eqref{eq:fiber_fubini} we obtain
\begin{align*}
\int_G  \left( \int_{q_1^{-1}(g)\cap C}  \omega'\right)  dg &= \int_C  q_1^* dg \wedge \omega'  \\							
    & =  \int_{N(K)\times N(L)}  p_*(q_1^* dg \wedge \omega')\\
& = \int_{N(K)\times N(L)}p^*_1\pi_2^* \phi \cdot p^*_2 \pi_2^*\psi\cdot   p_*q^*(dg \wedge \omega).
\end{align*}
It follows from Lemma~\ref{lem:invariance} that $p_*q^*(dg \wedge \omega)$ is a $\overline G \times \overline G$-invariant $(2n-2)$-form on $S\RR^n \times S\RR^n$.
Since every $\overline G \times \overline G$-invariant form on the product $S\RR^n\times S\RR^n$ is a sum of wedge products of $\overline G$-invariant forms on $S\RR^n$,
we obtain constants $c_{ij}$  such that 
$$ \int_{N(K)\times N(L)}p^*_1\pi_2^* \phi \cdot p^*_2 \pi_2^*\psi\cdot   p_*q^*(dg \wedge \omega) = \sum_{i,j}^{m} c_{ij} N(K)( \pi_2^*\phi \cdot \beta_i)\;   N(L)( \pi_2^*\psi \cdot \beta_j).$$
This establishes Theorem~\ref{thm:existence} in a special case.

Suppose now that $K$ and $L$ are general convex bodies and let $K_i\rightarrow K$ and 
$L_i\rightarrow L$ be sequences of convex bodies with smooth boundaries and such that all their principal curvatures positive. Since the collection $\{K_i+gL_i\colon i, g\}$ is contained in some sufficiently large ball,
 we have by Lemma~\ref{lem_massbound} and the definition of the comass of a differential form \cite{federer69}*{p. 358}
$$| N(K_i+gL_i)(\pi_2^*(\phi \; g\psi) \cdot \omega)| \leq \mathbf{M}(N(K_i+gL_i)) \mathbf{M}( \pi_2^*(\phi \; g\psi) \cdot \omega) \leq C$$
for some uniform constant $C>0$. 
For every $g\in G$ it follows from Lemma~\ref{lem_weakconvergence} that 
$$N(K_i+gL_i)(\pi_2^*(\phi \; g\psi) \cdot \omega)  \rightarrow N(K+gL)(\pi_2^*(\phi \; g\psi) \cdot \omega).$$
Hence, by the dominated convergence theorem, we obtain
$$\int_G N(K_i+gL_i)(\pi_2^*(\phi \; g\psi)\cdot  \omega )\; dg\rightarrow \int_G N(K+gL)( \pi_2^*(\phi \; g\psi) \cdot  \omega)\; dg.$$
This proves Theorem~1 for general convex bodies $K$ and $L$. 

Finally let $\phi$ and $\psi$ be bounded Baire functions. Then  $\phi$ and $\psi$ are the pointwise limits of sequences of uniformly bounded, smooth functions on $S^{n-1}$,
$$\phi_i(v)\rightarrow \phi(v)\qquad \text{and} \qquad \psi_i(v)\rightarrow \psi(v)$$
for every $v\in S^{n-1}$. Applying the dominated convergence theorem twice, we obtain 
$$\int_G N(K+gL)( \pi_2^*(\phi_i \; g\psi_i) \cdot  \omega )\; dg\rightarrow \int_G N(K+gL)(\pi_2^*(\phi \; g\psi)\cdot  \omega)\; dg.$$	
\end{proof}

\section{The algebra \texorpdfstring{$\Area^{G*}$}{Area*}}

We let $G \subset SO(n)$, $n\geq 2$, be a closed, connected subgroup acting transitively on the unit sphere and denote by $\calK (\RR^n)$ the space of convex bodies in $\RR^n$.

\subsection{General properties of the kinematic product}

A function which assigns to every convex body $K\subset \RR^n$ and Borel subset $U\subset S^{n-1}$ a real number is called a $G$-invariant area measure if there exists $\omega\in \Omega^{n-1}(S\RR^n)^G$ such that 
$$(K,U)\mapsto \int_{N(K)\cap \pi_2^{-1}(U) } \omega,$$
see \cite{wannerer13}. We denote by $\Area^G$ the space of all $G$-invariant area measures. In this terminology, Theorem~\ref{thm:existence} implies the following.

\begin{corollary}
 There exists a linear map $ A\colon  \Area^G \rightarrow \Area^G \otimes \Area^G$ called the \emph{local kinematic operator}
such that 
$$A(\Psi)(K,U; L, V)= \int_G \Psi(K+gL, U\cap gV) \; dg$$
for all convex bodies $K,L\subset\RR^n$ and Borel subsets $U,V\subset S^{n-1}$. 
\end{corollary}

The space of $\overline G$-invariant, continuous valuations is denoted by $\Val^G=\Val^G(\RR^n)$ and we write $\Val^{sm}$ the space of smooth, translation-invariant valuations.
We refer the reader to \cite{bernig_fu06} and the references therein  for more information on translation-invariant scalar valuations.

The globalization map is the surjective linear map $\glob\colon \Area^G\rightarrow \Val^G$ given by
$$\glob(\Psi)= \Psi(\;\cdot\; , S^{n-1}).$$
We denote by
$$ \bar a= (\id\otimes \glob)\circ A\colon  \Area^G \rightarrow \Area^G \otimes \Val^G$$
the semi-local kinematic operator. The additive kinematic operator is the linear map
$$ a\colon \Val^G \rightarrow \Val^G\otimes \Val^G$$
satisfying	
$$a(\phi)(K,L)= \int_G \phi(K+gL)\; dg,$$
see \cite{bernig_fu06}. The various kinematic operators fit in the following commutative diagram:
\begin{center}
 \begin{tikzpicture}
  \matrix (m) [matrix of math nodes,row sep=3em,column sep=4em,minimum width=2em] {
    \Area^G & \Area^G\otimes \Area^G \\
    \Area^G & \Area^G\otimes \Val^G \\
    \Val^G & \Val^G\otimes \Val^G \\};
  \path[-stealth]
    (m-1-1) edge node [right] {$\id $} (m-2-1)
            edge node [above] {$A$} (m-1-2)
    (m-2-1) edge node [right] {$\glob $} (m-3-1)
            edge node [above] {$\bar a$} (m-2-2)
    (m-1-2) edge node [right] {$\id \otimes \glob$} (m-2-2)
    (m-2-2) edge node [right] {$\glob \otimes \id$} (m-3-2)
    (m-3-1) edge node [above] {$a$} (m-3-2);
\end{tikzpicture}
\end{center}

We denote by $\phi \cdot \psi$ the product \cite{alesker04}, by $\phi*\psi$ the convolution \cite{bernig_fu06} and by $\hat \phi$ the Fourier transform \cite{alesker11} of $\phi,\psi\in\Val^{sm}$. The Fourier transform
has the fundamental property $\widehat{\phi \cdot \psi} = \hat \phi * \hat \psi$. 
The Poincar\'e duality map \cites{alesker04, bernig_fu06} is denoted by $\PD\colon \Val^{sm} \rightarrow \left(\Val^{sm}\right)^*$. 
Let  $\chi^*\colon \Val\rightarrow \RR$ be the linear functional $\left\langle \chi^*,\phi\right\rangle = \phi(\{0\})$, $0\in\RR^n$. It was shown in \cite{bernig_fu06} that 
\begin{equation}\label{eq:poincare_pairing_even}\left\langle \PD(\phi), \psi\right\rangle = \left\langle \chi^*, \phi* \psi\right\rangle\end{equation}
for even valuations $\phi,\psi\in\Val^{sm}$. We will see later (Proposition~\ref{lem:poincare}) that
\begin{equation}\label{eq:poincare_pairing}\left\langle \PD(\phi), \psi\right\rangle = -\left\langle \chi^*, \phi* \psi\right\rangle\end{equation}
if $\phi, \psi\in\Val^{sm}$ are odd valuations. It was proved in \cite{bernig_fu11} that every valuation in $\Val^G$ is even.

We define $\mu_C^G\in \Val^G$ by $\mu_C^G=\int_G \vol( \;\cdot \; + gC) \; dg = a(\vol)(\;\cdot\;,C)$.

\begin{proposition}\label{prop:poincare_evaluation}
$$\left\langle \PD(\mu_C^G), \phi\right\rangle = \phi(C)$$
for every $\phi\in\Val^G$. In particular, $\Val^G$ is spanned by $\mu_C^G$, $C\in \calK(\RR^n)$.
\end{proposition}
\begin{proof}
 This is proved as Proposition~2.17 in \cite{bernig_etal13}.  It follows from the definition of $a$ that
$$[(\chi^* \otimes \id) \circ a](\phi)= a(\phi)(\{0\}, \; \cdot\;) = \phi$$
Hence Fubini's theorem implies
$$\left\langle \PD(\mu_C^G), \phi\right\rangle = \left\langle \chi^*, \mu_C^G* \phi \right\rangle =\left\langle \chi^*, a(\phi)(\;\cdot\;,C)\right\rangle =\phi(C).$$
\end{proof}

The module product of \cite{wannerer13} is denoted by $\bar m\colon  \Val^G\otimes \Area^G\rightarrow \Area^G$, $\bar m (\phi, \Psi) = \phi * \Psi$. If $C$ is a convex body and $\Psi\in \Area^G$, then
$$\mu_C^G * \Psi (K,U) = \int_G \Psi(K+gC, U)\;dg$$
for all convex bodies $K$ and Borel set $U\subset S^{n-1}$. 

\begin{lemma}
 $A$, $\bar a$, and $a$ are all compatible with convolution with elements from $\Val^G$, i.e.\ 
\begin{align}A(\phi * \Psi) 	&= (\phi \otimes \vol) * A(\Psi) = (\vol \otimes \phi) * A(\Psi),\notag\\
	     \bar a(\phi * \Psi) 	&= (\phi \otimes \vol) * \bar a( \Psi) = (\vol \otimes \phi) * \bar a (\Psi),\notag\\
	      a(\phi * \psi)		&= (\phi \otimes \vol) * a(\psi) = (\vol \otimes \phi) * a(\psi),\notag\\
	      \bar a (\phi *\Psi) 	&= a(\phi) * (\Psi  \otimes \vol), \label{eq:semilocal}
\end{align}
whenever $\phi,\psi\in\Val^G$ and $\Psi\in \Area^G$.
\end{lemma}
\begin{proof}
This can be proved as Theorem~2.19 in \cite{bernig_etal13}. 
\end{proof}

The semi-local kinematic operator is in the following sense dual to the module product.
\begin{lemma}\label{lem:semilocal} Let $\bar m\colon \Hom(\Val^{G} \otimes \Area^{G},  \Area^{G}) = \Hom (\Area^{G}, \Area^{G}\otimes \Val^{G*} )$ be the module product. Then
 $$\bar m = (\operatorname{id} \otimes \PD) \circ \bar a.$$
\end{lemma}
\begin{proof}
 Using \eqref{eq:semilocal}, equation (15) of \cite{wannerer13}, the definition of $\mu_C^G$, and Proposition~\ref{prop:poincare_evaluation}, we obtain
$$\mu_C^G * \Psi (K,U)  = \bar a(\vol * \Psi)(K,U; C) = \bar a (\Psi)(K,U; C) = \left[(\id \otimes \PD) \circ \bar a\right] (\Psi)(K,U; \mu_C^G)$$
for every convex body $K$ and Borel set $U\subset S^{n-1}$. 
\end{proof}

Recall that a co-algebra consists of a vector space $X$ over some field $\KK$, a linear map $C\colon X\rightarrow  X\otimes X$ and a linear functional $\varepsilon \colon X\rightarrow \KK$ satisfying
$$(\id_X \otimes C)\circ C= (C\otimes \id_X)\circ C \qquad \text{(co-associativity)}$$
and 
$$(\varepsilon \otimes \id_X)\circ C=\id_X= (\id_X \otimes \varepsilon)\circ C.$$
The map $C$ is called a co-product and $\varepsilon$ is called a co-unit. A co-algebra is called co-commutative if $i\circ  C = C$,
where $i\colon X\otimes X \rightarrow X\otimes X$ denotes the interchange map $i(x\otimes y) = y\otimes x$.

\begin{proposition} \label{prop:co-algebra}
$(\Area^G, A,\chi^*\circ \glob)$ is a co-commutative co-algebra.
\end{proposition}
\begin{proof}
 Let $K,L,M\subset \RR^n$ be convex bodies, let $U,V,W\subset S^{n-1}$ be Borel sets and let $\Psi\in \Area^G$. By the invariance of the Haar measure,
\begin{align*} \left[(\id \otimes A)\circ A\right](\Psi)(K,U; L,V; M,W) &=\int_G \left( \int_G \Psi(K +g(L +hM), U \cap g (V \cap h W)) \; dg\right) \; dh \\
    & = \int_G \left( \int_G \Psi(K + gL + h M , U \cap g V \cap h W) \; dg\right) \; dh \\
    & = \left[(A \otimes \id)\circ A\right](\Psi)(K,U; L,V; M,W).
\end{align*}
Thus $A$ is co-associative. 
Since
$$\left[(\chi^*\circ \glob\otimes  \id )\circ A\right] (\Psi)(K,U) = A(\Psi)(\{0\},S^{n-1}; K,U)= \Psi(K,U)=  \left[(\id \otimes\chi^*\circ \glob )\circ A \right] (\Psi)(K,U),$$
$\chi^*\circ \glob$ is a co-unit. The co-commutativity of $A$ follows immediately from the invariance of the Haar measure.
\end{proof}

If $(X,C,\varepsilon)$ is a co-algebra and $X$ is finite-dimensional, then $(X^*,C^*,\varepsilon)$ is an algebra with unit $\varepsilon$. If $C$ is co-commutative, then $C^*$ is commutative. 
In particular, we obtain from Proposition~\ref{prop:co-algebra}
that $(\Area^{G*},A^*,\chi^*\circ \glob)$ is a commutative algebra. 
\begin{definition}
 We call $A^*\colon \Area^{G*}\otimes \Area^{G*}\rightarrow \Area^{G*}$ the kinematic product on 
$\Area^{G*}$.  We will also write $\Lambda_1 * \Lambda_2$ instead of $A^*(\Lambda_1\otimes \Lambda_2)$ for $\Lambda_1,\Lambda_2\in \Area^{G*}$
\end{definition}

\begin{lemma}\label{lem:kinprod_module}
Let $\Lambda\in \Area^{G*}$, $\phi\in \Val^G$, and put $m_\phi(\Phi):= \phi * \Phi$. Then
\begin{equation}\label{eq:kinprod}\Lambda * \glob^*(\PD(\phi)) = m_\phi^*(\Lambda).\end{equation}

\end{lemma}
\begin{proof} Using Lemma~\ref{lem:semilocal}, we compute
 \begin{align*}
  \left\langle m_\phi^*(\Lambda), \Phi\right\rangle  & = \left\langle  \Lambda , m_\phi(\Phi) \right\rangle = \left\langle \Lambda \otimes  \PD(\phi) , \bar a (\Phi)\right\rangle\\
     & = \left\langle \Lambda \otimes \glob^*(\PD(\phi)) , A(\Phi)\right\rangle =  \left\langle A^*( \Lambda \otimes \glob^*(\PD(\phi))) , \Phi\right\rangle\\
      & =  \left\langle  \Lambda * \glob^*(\PD(\phi)) , \Phi\right\rangle.
 \end{align*}

\end{proof}

\begin{corollary} \label{cor:subalgebra}
 If $\phi, \psi\in \Val^G$, then 
$$\glob^*(\PD(\phi)) * \glob^*(\PD(\psi)) = \glob^*(\PD(\phi * \psi)).	$$
In particular, the image of $\glob^*$ is a subalgebra of $\Area^{G*}$.
\end{corollary}

\subsection{The module of unitary area measures}
For the convenience of the reader we collect for quick reference those definitions and results of \cite{wannerer13} which are most important for us in the following.

Recall that the forms $\alpha, \beta, \gamma, \theta_0, \theta_1, \theta_2$, and  $\theta_s$, first defined in  \cite{bernig_fu11}, generate the algebra of unitarly invariant forms on the sphere bundle
$S\CC^n$. For all integers $0\leq k\leq 2n-1$ and $q\geq 0$ we denote by 
\begin{align*}B_{k,q}, &  \qquad  \max\{0,k-n\}\leq q <  k/2, \\    
   \Gamma_{k,q},& \qquad \max\{0,k-n+1\}\leq q \leq   k/2 ,
\end{align*}
the unitary area measures represented by the forms $\beta_{k,q}$ and $\gamma_{k,q}$. The collection of the $B_{k,q}$ and $\Gamma_{k,q}$ area measures is a basis of $\Area^{U(n)}$. Another very useful basis is given by
\begin{equation}\label{eq:DeltaN_basis}
    \begin{aligned} \Delta_{k,q} & = \frac{k-2q}{2n-k} B_{k,q} + \frac{2(n-k+q)}{2n-k} \Gamma_{k,q}, \qquad \max\{0,k-n\}\leq q \leq   k/2,\\
  N_{k,q} &= \frac{2(n-k+q)}{2n-k} \left( \Gamma_{k,q} - B_{k,q}\right), \qquad \max\{0,k-n+1\}\leq q <   k/2.
    \end{aligned}                                                                                                    
\end{equation}
Recall that 
\begin{equation}\label{eq:glob_DN}
 \glob(\Delta_{k,q})=\mu_{k,q} \qquad \text{and}\qquad \glob(N_{k,q})=0,
\end{equation}
where the valuations $\mu_{k,q}$ are the hermitian intrinsic volumes (see \cite{bernig_fu11}). 

The polynomials $p_k$ and $q_k$ are given explicitly by
\begin{equation} \label{eq:p_polynomial} p_k(s,t)= (-1)^k \sum_{i=0}^{\lfloor k/2 \rfloor} (-1)^i \binom{k-i}{i} s^i t^{k-2i}
\end{equation}
and 
\begin{equation} \label{eq:q_polynomial} q_k(s,t)= (-1)^{k+1} \sum_{i=0}^{\lfloor k/2 \rfloor} (-1)^i(i+1)  \binom{k+1-i}{i+1} s^i t^{k-2i}.
\end{equation}
They satisfy the relation
\begin{equation}\label{eq:fpq_relation}-(4s-t^2) q_{k-1} + tp_k = (k+1)^2 f_{k+1},\end{equation}
where $f_k$ denotes the Fu polynomial \cite{fu06}.

Recall also from \cite{fu06} that there are two special unitary valuations $s$ and $t$ which generate $ \Val^{U(n)}$ as an algebra. 
The module product on $\Area^{U(n)}$ satisfies the two fundamental relations
\begin{equation}\label{eq:BequalsG}
 p_n(\hat s, \hat t)* B_{2n-1,n-1} =q_{n-1}(\hat s, \hat t)* \Gamma_{2n-2,n-1}
\end{equation}
and
\begin{equation}\label{eq:Gequals0}
 p_n(\hat s, \hat t)* \Gamma_{2n-2,n-1}=0.
\end{equation}
Here $\hat s$ and $\hat t$ are the Fourier transforms of $s$ and $t$. Moreover, it was proved in \cite{wannerer13} that
\begin{equation}\label{eq:Gammasubmod}
 \textit{the span of the } \Gamma_{k,q} \textit{ coincides with the submodule of }\Area^{U(n)} \textit{ generated by }\Gamma_{2n-2,n-1}.
\end{equation}
and that
\begin{equation}\label{eq:Areamod}
 \textit{as a module } \Area^{U(n)} \textit{ is generated by }B_{2n-1,n-1} \textit{ and }\Gamma_{2n-2,n-1}.
\end{equation}

\subsection{The kinematic product in the unitary case} After the general considerations of the first subsection, we investigate now the case $\RR^{2n}=\CC^n$ and $G=U(n)$ in detail.  

In \cite{wannerer13} two special bases of $\Area^{U(n)}$ have been used:
The $B_{k,q}$-$\Gamma_{k,q}$ and the $\Delta_{k,q}$-$N_{k,q}$ basis. In the following we are going to use the corresponding dual bases for $\Area^{U(n)*}$ which, by \eqref{eq:DeltaN_basis},
 are related by
\begin{equation}\label{eq:B*equals}
 B_{k,q}^*=\frac{k-2q}{2n-k} \Delta_{k,q}^*  - \frac{2(n-k+q)}{2n-k} N_{k,q}^*, \qquad \max\{0,k-n\}\leq q <  k/2,
\end{equation}
and
\begin{equation}\label{eq:Gamma*equals}
 \Gamma_{k,q}^*=\left\{\begin{array}{ll} \frac{2(n-k+q)}{2n-k} (\Delta_{k,q}^* + N_{k,q}^*), &  \max\{0,k-n+1\}\leq q \leq   k/2 ;\\
					  \Delta_{k,q}^*, & 2q=k.
                       \end{array}\right.
\end{equation}
It is an immediate consequence of \eqref{eq:glob_DN} that the image of $\glob^*$ coincides with the span of the $\Delta_{k,q}^*$. Hence, by Corollary~\ref{cor:subalgebra}, the span of the
$\Delta_{k,q}^*$ is a subalgebra of $\Area^{U(n)*}$. 
We define  
$$\bar t := \glob^* (\PD (\hat t)) \qquad\text{and}\qquad     \bar s := \glob^* ( \PD (\hat s)).$$
For reasons which will become apparent later, we put
\begin{align*}
		    \bar v:= & \frac{2\omega_{2n-2}}{\omega_{2n-1}} B_{1,0}^*.
\end{align*}

From the definition of $\bar s$ and $\bar t$, Lemma~\ref{lem:kinprod_module}, and Propositions 4.7 and 4.8 of \cite{wannerer13}, we obtain the following formulas for multiplication by $\bar s$ and $\bar t$. 
\begin{lemma}\label{lem:sbartbar}
\begin{align} \bar t * B_{k,q}^*  & =  \frac{\omega_{2n-k}}{\pi \omega_{2n-k-1}} \left( (k-2q) B_{k+1,q+1}^*  + \frac{2(n-k+q)(k-2q)}{k-2q+1} B_{k+1,q}^*\right) \notag \\
	      \bar t * \Delta_{k,q}^* & = \frac{\omega_{2n-k}}{\pi \omega_{2n-k-1}}\left( (k-2q) \Delta_{k+1,q+1}^* + 2(n-k+q) \Delta_{k+1,q}^*\right) \notag \\ 
	      \bar t * N_{k,q}^* & = \frac{\omega_{2n-k}}{\pi \omega_{2n-k-1}} \frac{k-2q}{2n-k-1} \bigg( \Delta^*_{k+1,q+1} - \Delta^*_{k+1,q}        \label{eq:tbarN} \\
& \qquad \qquad \qquad + (2n-k) \left( N^*_{k+1,q+1} + \frac{2(n-k+q-1)}{k-2q+1} N_{k+1,q}^* \right)  \bigg) \notag  \\
	      \bar s* B_{k,q}^*  & = \frac{(k-2q)(k-2q-1)}{2\pi (2n-k)} B_{k+2,q+2}^* + \frac{2(n-k+q)(n-q)}{\pi (2n-k)} B_{k+2,q+1}^* \notag \\
	      \bar s * \Delta_{k,q}^* & =  \frac{(k-2q)(k-2q-1)}{2\pi (2n-k)} \Delta_{k+2,q+2}^* + \frac{2(n-k+q)(n-q)}{\pi (2n-k)} \Delta_{k+2,q+1}^* \notag \\
	      \bar s * N^*_{k,q}  & = \frac{(k-2q)(k-2q-1)}{2\pi (2n-k-2)} \left( N^*_{k+2,q+2}  + \frac{2}{2n-k} \Delta^*_{k+2,q+2} \right)  \notag \\ 
&  \qquad \qquad \qquad  + \frac{2(n-q)}{\pi (2n-k-2)}\left( (n-k+q-1)N^*_{k+2,q+1} - \frac{k-2q}{2n-k} \Delta^*_{k+2,q+1}\right) \notag
\end{align}

\end{lemma}
From the above we conclude that 
 $$\bar t = \frac{2 \omega_{2n-2}}{\omega_{2n-1}} \Delta^*_{1,0}=\frac{2 \omega_{2n-2}}{ \omega_{2n-1}}(B_{1,0}^* + \Gamma_{1,0}^*)\qquad \text{and}\qquad \bar s = \frac{n}{\pi} \Delta^*_{2,1}.$$

\begin{lemma}\label{lem:relations}
As elements of the algebra $\Area^{U(n)*}$ we have
$$p_n(\bar s,\bar t)-q_{n-1}(\bar s, \bar t)\bar v = 0 \qquad \text{and} \qquad p_n(\bar s,\bar t) \bar v=0.$$
\end{lemma}

\begin{proof}
 We prove $p_n(\bar s,\bar t) \bar v=0$ first. By \eqref{eq:Areamod}, any unitary area measure may be expressed as $ \phi * B_{2n-1,n-1} + \psi * \Gamma_{2n-2,n-1}$ for some $\phi,\psi \in \Val^{U(n)}$. 
Using \eqref{eq:kinprod}, \eqref{eq:BequalsG}, \eqref{eq:Gequals0},
 and \eqref{eq:Gammasubmod}, we obtain 
\begin{align*} \left\langle p_n(\bar s,\bar t) \bar v, \phi * B_{2n-1,n-1} + \psi * \Gamma_{2n-2,n-1}\right\rangle &= \left\langle \bar v, \phi* q_{n-1}(\hat s, \hat t) * \Gamma_{2n-2,n-1} \right\rangle \\
	  &= 0
\end{align*}

To prove the second identity, first observe that $p_n(\bar s,\bar t)$ is a linear combination of certain $B_{k,q}^*$. This follows immediately from \eqref{eq:Gequals0} and \eqref{eq:Gammasubmod}.
 At the same time $p_n(\bar s,\bar t)$ is also an element of span of the $\Delta_{k,q}^*$. Hence \eqref{eq:B*equals} implies that 
$p_n(\bar s,\bar t)$ is a multiple of $B_{n,0}^*$. 

Put $u=4s-t^2$. From \eqref{eq:fpq_relation} we obtain 
\begin{equation}\label{eq:vbaru}\left\langle q_{n-1}(\bar s,\bar t)\bar v,\hat u * \Psi\right\rangle= \left\langle -p_n(\bar s,\bar t) \bar v, \hat 	t*\Psi\right\rangle=0\end{equation}
for every unitary area measure $\Psi$. 

Next we claim that 
\begin{align}\label{eq:uBeta}
 &\textit{if }q_0> 0,\textit{ then --- modulo the subspace spanned by the } \Gamma_{k,q}\text{ --- }B_{n,q_0} \textit{ can be} \\ 
 & \textit{expressed as a linear combination of the area measures }\hat u * B_{n+2,q}. \notag
\end{align}
Indeed, by Corollary 3.8 of \cite{bernig_fu11}, $\mu_{n,q_0}$ can be expressed as a linear combination of the valuations $\hat u *\mu_{n+2,q}$ provided $q_0>0$. Since the globalization map is injective when restricted to the 
subspace spanned by the $\Delta_{k,q}$, we obtain that  $\Delta_{n,q_0}$ can be expressed as a linear combination of the area measures $\hat u *\Delta_{n+2,q}$ provided $q_0>0$. The claim follows now from \eqref{eq:Gammasubmod}.

Lemma~\ref{lem:sbartbar} implies that $q_{n-1}(\bar s,\bar t)\bar v$ is a linear combination of certain $B_{n,q}$. Together with \eqref{eq:vbaru} and \eqref{eq:uBeta}, 
we conclude that $q_{n-1}(\bar s,\bar t)\bar v$ is in fact a multiple of $B_{n,0}^*$.
Using $\hat s * B_{n,0}=0$, the explicit formulas for $p_n$ and $q_{n-1}$ given in \eqref{eq:p_polynomial} and \eqref{eq:q_polynomial}, and Lemma~\ref{lem:tnBn0} below, we compute
$$\left\langle p_n(\hat s,\hat t),B_{n,0}\right\rangle =(-1)^n \left\langle \Gamma_{0,0}^* ,   \hat t^{n} * B_{n,0}\right\rangle = (-1)^n\frac{2^n}{\omega_n}$$
and
$$\left\langle q_{n-1}(\hat s,\hat t)\bar v,B_{n,0}\right\rangle =(-1)^n \frac{2n \omega_{2n-2}}{\omega_{2n-1}} \left\langle B_{1,0}^* ,   \hat t^{n-1} * B_{n,0}\right\rangle = (-1)^n\frac{2^n}{\omega_n}.$$
\end{proof}

\begin{lemma}\label{lem:tnBn0}
 For $n\geq 2$, 
$$\hat t^{n-1} * B_{n,0}= \frac{2^{n-1}}{n} \frac{\omega_{2n-1}}{\omega_{2n-2} \omega_n} (B_{1,0} + (n-1) \Gamma_{1,0})$$
and 
$$\hat t^{n} *B_{n,0}=\frac{2^n}{\omega_n}\Gamma_{0,0}.$$
\end{lemma}
\begin{proof}
By Propositions~4.7 and 4.8 of \cite{wannerer13}, there exist numbers $c_i, d_i$ such that $\hat t ^i *B_{n,0} = c_i B_{n-i,0} + d_i \Gamma_{n-i,0}$ for $i=0,\ldots, n$. More precisely,
$$\begin{pmatrix} c_{i+1}\\ d_{i+1}\end{pmatrix} = \frac{2(i+1) \omega_{n+i+1}}{(n-i) \pi \omega_{n+i} } \begin{pmatrix} n-i-1 & 0 \\ 1 & n-i \end{pmatrix} \begin{pmatrix} c_{i}\\ d_{i}\end{pmatrix},\qquad i=0,\ldots, n-1$$
with $c_0=1$ and $d_0=0$. Since
$$ \begin{pmatrix} 1 & 0 \\ 1 & 2 \end{pmatrix} \cdots  \begin{pmatrix} n-1 & 0 \\ 1 & n \end{pmatrix} \begin{pmatrix} 1\\ 0\end{pmatrix}= (n-1)!\begin{pmatrix}1 \\ n-1\end{pmatrix}\qquad \text{for }n\geq 2,$$
the lemma follows.
\end{proof}

\begin{proposition}\label{prop:generators}

For every  $\Lambda\in \Area^{U(n)*}$ there exist polynomials $\phi, \psi$ in  $\bar s$ and $\bar t$ such that
\begin{equation}\label{eq:stv}\Lambda= \phi(\bar s,\bar t) + \psi(\bar s,\bar t) \bar v.\end{equation}
\end{proposition}
\begin{proof}
 Since the span of the $\Delta_{k,q}^* $ coincides with the image of $\glob^*$, it follows from Corollary~\ref{cor:subalgebra} that for every $\Delta_{k,q}^*$ 
there exists a polynomial $\phi$ in $\bar s$ and $\bar t$ such that
$\Delta_{k,q}^* = \phi(\bar s,\bar t)$. 
Since the case $n=2$ follows immediately from \eqref{eq:B*equals}, we assume from now on that $n\geq 3$. Suppose that $k=2m+1$ is odd and that $1\leq k< 2n-3 $ . Then $N_{2m+1,m}^* $ exists and, 
up to linear combinations of $\Delta_{k+1,q}^*$, by Lemma~\ref{lem:sbartbar}
$$\bar t * N_{2m+1,m}\equiv c N_{2m+2,m}$$
for some nonzero constant $c$. 
Hence if \eqref{eq:stv} holds for every $N_{k,q'}$, by \eqref{eq:tbarN},  \eqref{eq:stv} holds for every $N_{k+1,q}$ as well. 
Suppose now that $k=2m+2$ is even and that $2\leq k< 2n-3$. Then $N_{2m+1,m}$ and $N_{2m+3, m+1}$ exist and up to certain $\Delta_{k+1,q}^*$ 
$$\bar s * N_{2m+1,m}\equiv c N_{2m+3,1}$$
for some nonzero constant $c$. Hence if \eqref{eq:stv} holds for $N_{2m+1,m}^*$ and every $N_{k,q'}^*$, then, by \eqref{eq:tbarN}, \eqref{eq:stv} holds for all $N_{k+1,q}^* $ as well. This proves the proposition.
\end{proof}

Later, when we derive explicit kinematic formulas from our main theorem, the following two results will be useful.
\begin{lemma}\label{lem:Bideal}
For every $B_{k,q}^*$ there exists a polynomial $\psi=\psi(\bar s,\bar t)$ in $\bar s$ and $\bar t$ such that 
$$\psi(\bar s,\bar t) \bar v = B_{k,q}^*.$$
\end{lemma}
\begin{proof}
 By Proposition~\ref{prop:generators} there exist polynomials $\phi,\psi$ in $\bar s$ and $\bar t$ such that $\phi(\bar s,\bar t) + \psi(\bar s,\bar t)\bar v= B_{k,q}^*$. Since $\phi(\bar s,\bar t)$ is a linear combination of 
certain $\Delta_{k,q'}^*$ measures and $\psi(\bar s,\bar t)\bar v$ is a linear combination of certain $B_{k,q'}^*$  by Lemma~\ref{lem:sbartbar}, we conclude that if $\phi(\bar s, \bar t)\neq 0$, then necessarily $k\geq n$ and 
$\phi(\bar s,\bar t)$ is a multiple of $\Delta_{k,k-n}^*=B_{k,k-n}^*$. 
In the proof of Lemma~\ref{lem:relations}, we have shown that $q_{n-1}(\bar s,\bar t) \bar v$ is a nonzero multiple of $B_{n,0}^*=\Delta^*_{n,0}$ and hence $\bar t^{k-n} * q_{n-1}(\bar s,\bar t)\bar v$ is a nonzero multiple of
$B_{k,k-n}^*=\Delta_{k,k-n}^*$. Hence $B_{k,q}^*=\psi'(\bar s,\bar t)\bar v$ for some polynomial $\psi'$. 
\end{proof}

\begin{proposition}\label{prop:Delta_st}
For every $\Delta_{k,q}^*\in \Area^{U(n)*}$ we have
$$ \Delta_{k,q}^* =\frac{ \omega_{2n-k}(k-2q)!(n-k+q)!}{\pi^{n-k}2^{k-2q} n!} \sum_{i=q}^{\lfloor \frac{k}{2} \rfloor} \frac{(-1)^{i+q}}{(i-q)!} \frac{(n-i)!}{(k-2i)!  } \bar t ^{k-2i} \bar s^i. $$
\end{proposition}
\begin{proof}
By Corollary~\ref{cor:subalgebra} and since $\glob^*$ is injective, it is sufficient to check that 
\begin{equation}\label{eq:mustar_poincare}\left\langle \mu_{k,q}^*, \phi\right\rangle =
\frac{ \omega_{2n-k}(k-2q)!(n-k+q)!}{\pi^{n-k}2^{k-2q} n!} \sum_{i=q}^{\lfloor \frac{k}{2} \rfloor} \frac{(-1)^{i+q}}{(i-q)!} \frac{(n-i)!}{(k-2i)!  } (\hat t ^{k-2i} *\hat s^i  ,\phi)\end{equation}
for every $\phi\in\Val^{U(n)}$. 

For 
$$\phi= \hat t^{2n-k-2j} * \hat u^j,\qquad u=4s-t,\qquad 0\leq 2j\leq 2n-k ,$$ 
the left-hand side of \eqref{eq:mustar_poincare} evaluates to
\begin{equation}\label{eq:mustar_LHS} \left\langle \mu_{2n-k,n-k+q}^*,  t^{2n-k-2j}u^j \right\rangle = \frac{\omega_{2n-k} (2n-k-2j)! (2j)!}{\pi^{2n-k}} \binom{n-k+q}{j} 
\end{equation}
by Proposition~3.7 of \cite{bernig_fu11}.
Since 
$$t^{2n-2i} s^i (B(\CC^n)) = \binom{2n-2i}{n-i},$$
see \cite{fu06}, and $u=4s-t^2$, one shows by induction on $j$ that 
\begin{equation}\label{eq:poincare_stu}t^{2n-2i-2j} s^i u^j(B(\CC^n)) =\binom{2j}{j} \binom{2n-2i-2j}{n-i-j} \binom{n-i}{j}^{-1}.\end{equation}
Hence the right-hand side of \eqref{eq:mustar_poincare} becomes
\begin{equation}\label{eq:mustar_RHS}\frac{\omega_{2n-k}(n-k+q)!(2j)!}{\pi^{2n-k} j!}\cdot  \frac{(k-2q)!}{2^{k-2q}} \sum_{i=q}^{\lfloor \frac{k}{2} \rfloor}  \frac{(-1)^{i+q}}{(i-q)!} \frac{(2n-2i-2j)!}{(k-2i)!(n-i-j)!}.\end{equation}
Comparing \eqref{eq:mustar_LHS} and \eqref{eq:mustar_RHS}, we see that \eqref{eq:mustar_poincare} will be proved if we can show that
\begin{equation}\label{eq:combinat0}2^{k-2q} \binom{n-j-q}{k-2q} = \sum_{i=0}^{\lfloor \frac{k-2q}{2} \rfloor}  (-1)^{i} \binom{2n-2i-2j-2q}{k-2i-2q}\binom{n-j-q}{i}.\end{equation}
Now notice that \eqref{eq:combinat0} is just \eqref{eq:combinat2} with $m$ replaced by $n +q -j -k$ and $r$ replaced by $k-2q$. 
\end{proof}

\begin{lemma}
If $r$ is a nonnegative integer and $m$ is an integer satisfying $2m+r\geq 0$, then
\begin{equation}
\label{eq:combinat2}
 2^r \binom{m+r}{r}=  \sum^{\lfloor \frac{r}{2} \rfloor}_{i=0} (-1)^i \binom{2m+2r-2i}{r-2i}  \binom{m+r}{i}.
\end{equation}

\end{lemma}
\begin{proof}
We are going to use Zeilberger's algorithm \cite{zeilberger96}*{p. 101}. 
Fix $r\geq 0$. We denote the sum on the right-hand side of \eqref{eq:combinat2} by $S_m$ and put
$$F(m,i):= (-1)^i \binom{2m+2r-2i}{r-2i}  \binom{m+r}{i},\qquad 2m\geq -r.$$
One immediately checks that $F(m,i)$ satisfies the recurrence relation
\begin{equation}\label{eq:recurrence}-(m+r+1) F(m,i) + (m+1) F(m+1,i) = G(m,i+1)-G(m,i) \end{equation}
where 
$$G(m,i)= F(m,i) \frac{2i (2m+2r-2i+1)(m+r+1)}{(2m+r+1)(2m+r+2)}.$$
Summing the recurrence \eqref{eq:recurrence} over $i$ from $0$ to $\lfloor r/2 \rfloor$ and  using
that $ G(m,i+1)-G(m,i)$ telescopes to $0$, we obtain
$$ (m+1) S_{m+1} = (m+r+1) S_m$$
and therefore
$$S_m=\binom{m+r}{r} S_0,\qquad  2m\geq-r.$$
To show that  $S_0=2^r$ we put
$$f(r,i):=   (-1)^i \binom{2r-2i}{r}  \binom{r}{i}.$$
Then $f(r,i)$ satisfies 
$$(r+1)\left( 2 f(r,i)-f(r+1,i) \right)= g(r,i+1) -g(r,i)$$
with 
$$g(r,i)=4i \frac{2r-2i+1}{r-2i+1} f(r,i).$$ 
Summing the recurrence relation for $f$ over $i$ from   $0$ to $\lfloor r/2 \rfloor$, we obtain $S_0=2^r$.
\end{proof}

\section{Tensor valuations} \label{sec:tensor}

In the following it will be convenient to work with an abstract $n$-dimensional euclidean vector space $V$ instead of $\RR^n$. The group of special orthogonal transformations of $V$ will be denoted by $SO(V)$, 
the unit sphere by $S(V)$ and the sphere bundle of $V$ by $SV$. As before, $G\subset SO(V)$ is a closed, connected subgroup acting transitively on the unit sphere.
Given a non-negative integer $r$ we denote by $\Val^r=\Val^r(V)$ the vector space of translation-invariant, continuous valuations on $V$ with values in
$\Sym^r V$. Here $\Sym^rV\subset \bigotimes^r V$ denotes the subspace of symmetric tensors of rank $r$. We call $\Val^r$ the space of tensor valuations of rank $r$ and we denote by $\Val_k^r\subset \Val^r$ the subspace of $k$-homogeneous 
valuations. Applying McMullen's theorem \cite{mcmullen77} componentwise, we see that 
$$\Val^r= \bigoplus_{k=0}^n \Val_k^r.$$ 
Note that in this paper we only consider translation-invariant tensor valuations. 

If $f\colon V\rightarrow W$ is a linear map,
then $f^{\otimes r} \colon \bigotimes^r V\rightarrow \bigotimes^r W$ maps symmetric tensors to symmetric tensors. In particular, the action of $G$ on $V$ induces an action  on all
symmetric powers $\Sym^r V$. The group $G$ acts on tensor valuations of rank $r$ by $(g\cdot \Phi)(K)=g^{\otimes r} ( \Phi (g^{-1}K))$. A tensor valuation of rank $r$ is called  $G$-covariant, if 
$$g\cdot \Phi = \Phi$$
for every $g\in G$. The subspace of $G$-covariant tensor 
valuations is denoted by $\Val_k^{r,G}$. For more information on $SO(V)$-covariant tensor valuations, also in the non-translation-invariant case, see 
\cites{alesker99a, alesker_etal11, hug_etal07, hug_etal08,ludwig03, ludwig13, mcmullen97}. 

We denote by $\Val^{sm,r}\subset \Val^r$ the subspace of tensor valuations which can be written as
$\Phi(K)=\int_{N(K)} \omega,$
where $\omega$ is a translation-invariant, smooth differential form on the sphere bundle $SV$ with values in $\Sym^r V$.

\begin{lemma}
 $\Val^{r,G}\subset \Val^{sm,r}$.  In particular, $\Val^{r,G}$ is 
finite-dimensional. 
\end{lemma}
\begin{proof}
A left $G$-action on the space of translation-invariant smooth differential forms on $SV$ with values in $\Sym^r V$ is given by
\begin{equation}\label{eq:G-action}L_g(v\otimes \omega):= (g^{-1})^*\omega \otimes  gv\end{equation}
for every $g\in G$, $v\in \Sym^r V$, and every translation-invariant form $\omega\in \Omega(SV)$. We define projections to the spaces of $G$-covariant valuations and 
$G$-invariant forms by
$$\pi_G(\Phi)= \int_G g\cdot \Phi\; dg \qquad \text{and}\qquad \pi_G(\omega)=\int_G L_g \omega \; dg,$$
respectively. Applying Theorem 5.2.1 of \cite{alesker06} componentwise, we obtain tensor valuations $\Phi_i$ given by smooth, translation-invariant differential forms $\omega_i$ such that 
$\Phi_i \rightarrow \Phi$ uniformly on compact subsets. 
Since $G$ acts transitively on the unit sphere, the space of $G$-invariant, translation-invariant forms with in values in $\Sym^r V$ is finite-dimensional. Therefore the space of tensor valuations represented by $G$-invariant forms
 is finite-dimensional
and consequently closed. Since
$$\pi_G(\Phi_i)(K)= \int_{N(K)} \pi_G(\omega_i) $$
and  $\pi_G(\Phi_i)\rightarrow \Phi$ uniformly on compact subsets, we conclude that $\Phi$ is represented by a $G$-invariant, translation-invariant form. 
\end{proof}

\subsection{Algebraic structures for tensor valuations}

In this subsection we extend some of the algebraic structures for scalar valuations to tensor valuations.

Choose some orthonormal basis $\{e_i\}$ of $V$. If $a$ is a symmetric tensor of rank $r$, then there exist numbers $a_{i_1\ldots i_r}$ such that
$$a= \sum_{i_1,\ldots,i_r=1}^n a^{i_1\ldots i_r} e_{i_1}\otimes \cdots \otimes e_{i_r}.$$
Using the Einstein summation convention --- which we will do in the following --- this can be written as $a=  a^{i_1\ldots i_r} e_{i_1}\otimes \cdots \otimes e_{i_r}$.
 Since $a$ is a symmetric tensor
we have $a^{i_{\pi(1)} \ldots i_{\pi(r)}}= a^{i_1\ldots i_r}$
for every permutation $\pi$ of the numbers $1,\ldots,r$.  
If $b$ is a symmetric tensor of rank $s$, $b=b^{i_1\ldots i_s} e_{i_1}\otimes \cdots e_{i_s}$,
then 
$$ab= \Sym(a\otimes b)= \frac{1}{(r+s)!} \sum_\pi a^{i_{\pi(1)}\ldots i_{\pi(r)}} b^{i_{\pi(r+1)}\ldots i_{\pi(r+s) } } e_{i_{1}} \otimes  \cdots \otimes e_{i_{r+s}},$$
where the sum extends over all permutations of $1,\ldots, r+s$.
If $s\geq r$, then we define the contraction of $a$ with $b$ by
$$\contr(a,b)=\contr(b,a):= \sum_{i_1,\ldots,i_r=1}^n a^{i_1\ldots i_r} b^{i_1\ldots i_r j_1\ldots j_q} e_{j_1}\otimes \cdots \otimes e_{j_q},\qquad q=s-r.$$
Note that $\contr(a,b)$ is a symmetric tensor of rank $q$ and that the definition of $\contr(a,b)$ is independent of the choice of orthonormal basis. Moreover, we have
\begin{equation}\label{eq:innerprodcontr}(a,b)=\contr(a,b), \qquad a,b\in \Sym^r V,\end{equation}
where $(a,b)$ the denotes the restriction of the usual inner product on $\bigotimes^r V$ to $\Sym^r V$. 
If $a\in\Sym^r V$, $b\in \Sym^s V$, and $c$ is a symmetric tensor of rank at least $r+s$, then
\begin{equation}\label{eq:contractions}
\contr(a,\contr(b,c))=\contr(ab,c).
\end{equation}

From the discussion above it is clear how the convolution and contraction of tensor valuations should be defined. First of all,
 if $\Phi\in \Val^{r}$ is a tensor valuation of rank $r$, then there exist scalar valuations $\Phi^{i_1\ldots i_r}$ such that
$$\Phi(K)= \Phi^{i_1\ldots i_r}(K) e_{i_1}\otimes \cdots \otimes e_{i_r}$$
for every convex body $K\subset V$ and $\Phi^{i_{\pi(1)} \ldots i_{\pi(r)}}(K)= \Phi^{i_1\ldots i_r}(K)$
for every permutation $\pi$ of $1,\ldots,r$. We define the convolution of $\Phi\in \Val^{sm,r}$ and $\Psi\in \Val^{sm,s}$ by
$$\Phi * \Psi=  \frac{1}{(r+s)!} \sum_\pi \Phi^{i_{\pi(1)}\ldots i_{\pi(r)}} *\Psi^{i_{\pi(r+1)}\ldots i_{\pi(r+s) } } e_{i_{1}} \otimes  \cdots \otimes e_{i_{r+s}}\in\Val^{r+s},$$
where the sum extends over all permutations of $1,\ldots, r+s$. The contraction of $\Phi$ with $\Psi$ defined by
$$\contr(\Phi,\Psi)=\contr(\Psi,\Phi):= \sum_{i_1,\ldots,i_r=1}^n \Phi^{i_1\ldots i_r}* \Psi^{i_1\ldots i_r j_1\ldots j_q} e_{j_1}\otimes \cdots \otimes e_{j_q} \in\Val^{sm,q},\qquad q=s-r.$$
Note that convolution and contraction of tensor valuations are compatible with the action of $G$, i.e.\ 
\begin{equation}\label{eq:Ginvariantcontra}g\cdot (\Phi *\Psi) = (g\cdot \Phi)*(g\cdot \Psi)\qquad\text{and}\qquad  \contr(g\cdot \Phi, g\cdot \Psi) = g\cdot \contr(\Phi,\Psi)  \end{equation}
for every $g\in G$. In particular, $\Phi* \Psi$ is $G$-covariant, if $\Phi$ and $\Psi$ are. 
Moreover, whenever the rank of $\Phi$ is greater than or equal to the sum of the ranks of $\Psi_1$ and $\Psi_2$  
\begin{equation}\label{eq:valcontraconv} \contr(\Psi_1,\contr(\Psi_2,\Phi))= \contr(\Psi_1 *\Psi_2, \Phi)\end{equation}
can be proved as \eqref{eq:contractions}.
If $r=s$,  we define the Poincar\'e duality paring of $\Phi$ and $\Psi$ by
$$(\Phi,\Psi)= \left\langle \chi^*, \contr(\Phi,\Psi)\right\rangle$$
and the Poincar\'e duality map $\widehat{\PD}\colon \Val^{sm,r} \to \left( \Val^{sm, r}\right)^*$ by $\left\langle \widehat{\PD}(\Phi),\Psi\right\rangle = (\Phi,\Psi)$.

We establish now a formula for the Poincar\'e duality paring of tensor valuations similar to Theorem~4.1 of \cite{bernig09}. To state the result we first have to introduce some notation.
For differential forms on a manifold $M$ with values in a finite-dimensional euclidean vector space $V$ we define a pairing by
$$(\;\cdot\;,\; \cdot\;) \colon \Omega(M,  V) \otimes \Omega(M,V) \rightarrow \Omega(M)$$
by
$$(\omega\otimes v, \omega'\otimes v):= \left\langle v,v'\right\rangle\; \omega \wedge \omega'.$$
The pull-back and exterior derivative of $V$-valued forms are defined componentwise. For differential forms on $M$ with values in the algebra $\Sym V$ we define a wedge product 
$$\wedge\  \colon \Omega(M,  \Sym V) \otimes \Omega(M,\Sym V) \rightarrow \Omega(M, \Sym V)$$
by
$$(\omega\otimes v) \wedge  (\omega'\otimes v'):=  \omega \wedge \omega' \otimes vv'.$$

Let us also recall a few facts regarding the convolution of smooth, translation-invariant valuations \cite{bernig_fu06}.
If $\phi_1,\phi_2\in\Val^{sm}$ are given by the  translation-invariant  differential forms $\omega_1$, $\omega_2$, then the convolution product $\phi_1 *\phi_2$ is represented by the differential form
\begin{equation}\label{eq:convproduct}*_1^{-1} (*_1\omega_1 \wedge *_1 D \omega_2), \end{equation}
where $D$ denotes the Rumin differential (see \cite{rumin94}), 
$$*_1 ( \pi_1^*\gamma_1\wedge \pi_2^*\gamma_2) = (-1)^{\binom{n-\operatorname{deg} \gamma_1}{2}} \pi_1^*(*\gamma_1)\wedge  \pi_2^*\gamma_2,$$
and $*$ denotes the Hodge star operator on $V$. If we let $D$ and $*_1$ act on vector-valued forms componentwise, then formula \eqref{eq:convproduct} holds verbatim for tensor valuations.

\begin{proposition} \label{lem:poincare} Let $0<k<n$. If $\Phi \in \Val_k^r$ and $\Phi' \in \Val^{r}_{n-k}$ are tensor valuations represented by differential forms $\omega$ and $\omega'$, then 
\begin{equation}\label{eq:poincare} (\Phi, \Phi')= \frac{(-1)^{k} }{\omega_n} \int_{B(V) \times S(V)} (\omega,D \omega'),\end{equation}
where $B(V)$ denotes the unit ball of $V$.
\end{proposition}

\begin{proof} By the bilinearity of the pairings $(\Phi, \Phi')$ and $(\omega,D \omega')$ it will be sufficient to prove the formula for scalar valuations.

Let $x_1,\ldots, x_n$ be the standard coordinates on $V=\RR^n$ and let $y_1,\ldots,y_{n-1}$ be local coordinates on $S^{n-1}$ Put $dx_I=dx_{i_1} \wedge \cdots \wedge dx_{i_k}$ and 
$dx_{I'}= dx_{i'_1}\wedge \cdots \wedge dx_{i'_{n-k}}$ and put $dy_{J}=dy_{j_1}\wedge \cdots \wedge dy_{j_{n-k-1}}$ and $dy_{J'}=dy_{j'_1}\wedge \cdots \wedge dy_{j'_{k}}$. By linearity it will be sufficient to 
prove 
$$*_1^{-1}(*_1(dx_I \wedge dy_J) \wedge *_1 (dx_{I'}\wedge dy_{J'})) =(-1)^k \left(\frac{\partial }{\partial x_1} \wedge \cdots \wedge\frac{\partial }{\partial x_n}\right) \lrcorner  (dx_I \wedge dy_J \wedge dx_{I'}\wedge dy_{J'}). $$

If $dx_I\wedge dx_{I'}=0$ there is nothing to prove since the left-hand and right-hand side are zero. If $dx_I\wedge dx_{I'}\neq 0$, we choose $\epsilon\in \{-1,1\}$ such that
$dx_I\wedge dx_{I'} =\epsilon \vol_{\RR^n}$.  Then $*(dx_I)= \epsilon dx_{I'}$ and $*(dx_{I'})= (-1)^{k(n-k)}\epsilon dx_{I}$. Using that
$$\binom{n}{2} + \binom{k}{2} + \binom{n-k}{2} \equiv k(n-1) \mod 2$$
we compute 
\begin{align*}
*_1^{-1}( *_1(dx_I \wedge dy_J) \wedge *_1 (dx_{I'}\wedge dy_{J'})) &= *_1^{-1}\left( (-1)^{\binom{k}{2} + \binom{n-k}{2} + k(n-k)} dx_{I'} \wedge dy_J \wedge dx_I \wedge dy_{J'} \right) \\
    & = *_1^{-1}\left( (-1)^{\binom{k}{2} + \binom{n-k}{2} +k(n-k-1)} dx_I\wedge dx_{I'} \wedge dy_J \wedge dy_{J'}\right)\\
     &= *_1^{-1}\left( (-1)^{\binom{k}{2} + \binom{n-k}{2} +k(n-k-1)} \epsilon \vol_{\RR^n}  \wedge dy_J \wedge dy_{J'}\right)\\
      &= (-1)^{k^2} \epsilon\; dy_J\wedge dy_{J'}\\
      & =  (-1)^{k}\epsilon\; dy_J\wedge dy_{J'}.
\end{align*}
On the other hand,
\begin{align*}
\left(\frac{\partial }{\partial x_1} \wedge \cdots \wedge\frac{\partial }{\partial x_n}\right) \lrcorner  (dx_I \wedge dy_J \wedge dx_{I'}\wedge dy_{J'}) 
&= (-1)^{(n-k)(n-k-1)} \epsilon\;  dy_J  \wedge dy_{J'}  \\
    & =  \epsilon\;  dy_J  \wedge dy_{J'}.
\end{align*}

\end{proof}

We remark that Proposition~\ref{lem:poincare} together with Theorem~4.1 of \cite{bernig09} implies \eqref{eq:poincare_pairing}.

\subsection{The ftaig for tensor valuations}

The goal of this subsection is to introduce a new kinematic operator for tensor valuations and to prove a 
version of the fundamental theorem of algebraic integral geometry (ftaig) for this operator.

Observe that if  $f,g\colon V\rightarrow W$ are linear maps, $r_1,r_2\geq 0$, and $p\in \Sym^{r_1+r_2}V$, then $f^{\otimes r_1} \otimes g^{\otimes r_2} (p)$ is in general not an element of $\Sym^{r_1+r_2}W$, but only
an element of $ \Sym^{r_1}W \otimes \Sym^{r_2}W$. Given $\Phi\in  \Val^{r_1+r_2,G}_k$, we define  a bivaluation (see, e.g., \cite{ludwig10} or \cite{alesker_etal11} for more information on bivaluations) 
with values in $ \Sym^{r_1}V \otimes \Sym^{r_2}V$ by
$$a^{r_1,r_2}(\Phi)(K,L) = \int_G (\operatorname{id}^{\otimes r_1} \otimes g^{\otimes r_2}) \Phi(K+ g^{-1}L)\; dg$$
and call $a^{r_1,r_2}$ the \emph{additive kinematic operator} for tensor valuations. 

Note that 
\begin{equation}
 \label{eq:Gcovar}
  a^{r_1,r_2}(\Phi)(h_1K,h_2 L)= h_1^{\otimes r_1}\otimes  h_2^{\otimes r_2} (A_G^{r_1,r_2}(\Phi)(K,L))
\end{equation}
whenever $h_1,h_2\in G$.

\begin{theorem}
 Let $\Phi_1,\ldots, \Phi_{m_1}$ be a basis of $\Val^{r_1,G}$ and let $\Psi_1,\ldots, \Psi_{m_2}$ be a basis of $\Val^{r_2,G}$. If $\Phi\in \Val^{r_1+r_2,G}$, then 
there exist constants $c_{ij}$ depending only on $\Phi$ such that
$$a^{r_1,r_2}(\Phi)(K,L) = \sum_{i,j} c_{ij} \; \Phi_i(K)\otimes \Psi_j(L)$$
for all convex bodies $K,L\subset V$.
In particular, we can consider the additive kinematic operator  as a linear  map 
$$a^{r_1,r_2}\colon \Val^{r_1+r_2,G}\rightarrow\Val^{r_1,G}\otimes \Val^{r_2,G}.$$
\end{theorem}

\begin{proof}
 Let $\{a_r\} $  be a basis of $ \Sym^{r_1}V $ and let  $\{b_s\}$ be a basis of $ \Sym^{r_2}V $. For every $K$ and $L$ there exist numbers $\phi_{rs}(K,L)$ such that 
$$a^{r_1,r_2}(\Phi)(K,L) = \sum_{r,s} \phi_{rs}(K,L) a_r\otimes b_s.$$
Since $\{ a_r\otimes b_s\}$ is a basis of $ \Sym^{r_1}V \otimes \Sym^{r_2}V$, $\phi_{rs}$ is a bivaluation. If we fix $L$, then by \eqref{eq:Gcovar}, $K\mapsto \sum_{r} \phi_{rs}(K,L) a_r$
is an element of $\Val^{r_1, G}$ for every $s$ and hence there exist numbers $\mu_{is}(L)$ such that
 $$\sum_{r} \phi_{rs}(K,L) a_r = \sum_i \mu_{is}(L) \Phi_i(K).$$
Since $\{ \Phi_i\}$ is a basis of $\Val^{r_1,G}$, each $\mu_{is}$ is a valuation. Rearranging terms, we arrive at
$$a^{r_1,r_2}(\Phi)(K,L) = \sum_{i,s}\mu_{is}(L) \Phi_i(K) \otimes b_s= \sum_i \Phi_i(K) \otimes \left(\sum_s \mu_{is}(L) b_s\right).$$
Again  by \eqref{eq:Gcovar}, for each $i$, $L\mapsto \sum_s \mu_{is}(L) b_s$
is an element of $\Val^{r_2,G}$. Hence there exist constants $c_{ij}$ depending only on $\Phi$ such that 
$$\sum_s \mu_{is}(L) b_s=\sum_j c_{ij} \Psi_j(L)$$
and thus
$$a^{r_1,r_2}(\Phi)(K,L)= \sum_{i,j} c_{ij} \; \Phi_i(K)\otimes \Psi_j(L).$$

\end{proof}

\begin{lemma}\label{lem:kinmatcontraction}
 If $\Psi_1\in\Val^{r_1,G}$ and $\Psi_2\in \Val^{r_2,G}$, then 
$$ \contr(\Psi_1,\;\cdot \;)\otimes  \contr(\Psi_2,\;\cdot \;) \circ a^{r_1,r_2} = a\circ \contr(\Psi_1*\Psi_2, \; \cdot\;).$$
\end{lemma}
\begin{proof}
If $C\subset V$ is a convex body with smooth boundary and all principal curvatures positive, then $\mu_C(K)=\vol(K+C)$ is a smooth valuation and 
$\mu_C *\phi(K)=\phi(K+C)$ for every $\phi\in \Val^{sm}$ and every convex body $K\subset V$, see \cite{bernig_fu06}.

Assume for the moment that $L$ has smooth boundary with all principal curvatures positive. Then
\begin{align*} a^{r_1,r_2}(\Phi)(\; \cdot\;,L) &= \int_G (\id^{\otimes r_1}\otimes g^{\otimes r_2})  \mu_{g^{-1}L}*\Phi \; dg \\
 & = \int_G \mu_{g^{-1}L} * \Phi_{i_1\ldots i_{r_1+r_2}} e_1\otimes\cdots\otimes e_{i_{r_1}}\otimes ge_{i_{r_1+1}} 
\otimes \cdots \otimes ge_{i_{r_1+r_2}} \; dg
\end{align*}
and hence 
\begin{align*}
\left[ (\contr(\Psi_1,\;\cdot \;)\otimes  \id) \circ a^{r_1,r_2}\right] (\Phi)(K, L)  &= \int_G \mu_{g^{-1}L} * (\Psi_1)_{i_1\ldots i_{r_1}}*\Phi_{i_1\ldots i_{r_1+r_2}}  ge_{i_{r_1+1}} \otimes \cdots \otimes ge_{i_{r_1+r_2}} \; dg\\
			  &= \int_G g^{\otimes r_2 } \contr(\Psi_1,\Phi)(K+g^{-1} L)\; dg\\
			  &= \int_G  \contr(\Psi_1,\Phi)(gK+L)\; dg,
\end{align*}
where the last line follows from \eqref{eq:Ginvariantcontra}. 
Applying now \eqref{eq:valcontraconv}, yields 
$$
 \left[ \left(\contr(\Psi_1,\;\cdot \;)\otimes  \contr(\Psi_2,\;\cdot \;)\right) \circ a^{r_1,r_2}\right](\Phi)(K,L) = \int_G \contr(\Psi_1*\Psi_2,\Phi))(gK+L).$$
By continuity, this equality holds for all convex bodies $K$ and $L$.

\end{proof}

\begin{lemma}
The Poincar\'e duality map $\widehat{\PD}\colon \Val^{r,G} \rightarrow (\Val^{r,G})^*$ is a linear isomorphism.
\end{lemma}
\begin{proof} Fix some nonzero tensor valuation $\Phi\in \Val^{r,G}$. It follows from  \eqref{eq:poincare_pairing_even} and \eqref{eq:poincare_pairing} that the 
pairing of scalar valuations $\left\langle \chi^*, \phi* \psi\right\rangle$ is perfect.  Hence there exists a tensor valuation $\Psi'$, not necessarily $G$-covariant, 
such that
$(\Phi, \Psi')\neq 0$. By \eqref{eq:Ginvariantcontra}, the Poincar\'e pairing of tensor valuations is $G$-invariant. Hence, if we put $\Psi=\int_G g\cdot \Psi' \; dg$, then
$$(\Phi , \Psi) =\int_G ( \Phi , g\cdot \Psi')\; dg = ( \Phi , \Psi')\neq 0.$$
\end{proof}

\begin{theorem}[ftaig]\label{thm:ftaig} Let $r_1,r_2$ be non-negative integers and denote by $c_G\colon \Val^{r_1,G}\otimes \Val^{r_2,G}\rightarrow \Val^{r_1+r_2, G}$ the convolution of tensor valuations. Then
$$a^{r_1,r_2} = (\widehat{\PD}^{-1}\otimes \widehat{\PD}^{-1}) \circ c^*_G \circ \widehat{\PD}.$$
\end{theorem}
\begin{proof}
 Let $\Psi_1$ and $\Psi_2$ be $G$-covariant  tensor valuations of rank $r_1$ and $r_2$, respectively. By the definition of Poincar\'e duality and Lemma~\ref{lem:kinmatcontraction} we have
\begin{align*} \left\langle \widehat{\PD} \otimes \widehat{\PD} \circ a^{r_1,r_2}(\Phi), \Psi_1\otimes \Psi_2\right\rangle & = \left\langle \contr(\Psi_1, \; \cdot\; ) \otimes \contr(\Psi_2, \; \cdot\; ) \circ a^{r_1,r_2}(\Phi), \chi^*\otimes \chi^*\right\rangle\\
 & = \left\langle a(\contr(\Psi_1 * \Psi_2 , \Phi)) , \chi^* \otimes \chi^* \right\rangle.
\end{align*}
On the other hand, 
$$\left\langle c_G^* \circ \widehat{\PD}(\Phi), \Psi_1\otimes \Psi_2\right\rangle = \left\langle \widehat{\PD}(\Phi) , \Psi_1 *\Psi_2\right\rangle = \left\langle \contr(\Psi_1*\Psi_2, \Phi),\chi^*\right\rangle.$$
The theorem follows now from the fact that $\left\langle a(\phi),\chi^*\otimes \chi^*\right\rangle = \left\langle \phi , \chi^*\right\rangle$ for every $\phi\in \Val^G$ which is clear from the definition of $a$ and $\chi^*$.
\end{proof}

\subsection{Moment maps}

For each $r\geq 0$ we define the $r$th order moment map $M^r\colon \Area^G\rightarrow \Val^{ r,G}$ by
$$M^r(\Phi)(K) = \int_{S(V)} u^{ r} \; d\Phi(K,u).$$
Note that in particular  $M^0=\glob$. The reason why  moment maps are useful for us is that they connect the kinematic operator for area measures with the additive kinematic operators for tensor valuations.

\begin{proposition}\label{prop:tensorkin}
If $r_1, r_2$ are non-negative integers and $\Phi\in \Area^G$, then 
$$\left[ a^{r_1,r_2} \circ M^{r_1+r_2}\right](\Phi) =\left[ (M^{r_1}\otimes M^{r_2}) \circ A \right](\Phi).$$

\end{proposition}
\begin{proof}
Let $K,L\subset V$ be convex bodies. Using the notation of Theorem~\ref{thm:existence}, we obtain
$$\left[ (M^{r_1}\otimes M^{r_2}) \circ A\right](\Phi)(K,L)= \int_G \Phi(K+g^{-1} L, u^{r_1} \otimes (gu)^{r_2} ) \; dg.$$
On the other hand, since $u^{r_1} \otimes (gu)^{r_2} = \operatorname{id}^{\otimes r_1} \otimes g^{\otimes r_2}(u^{r_1} \otimes u^{r_2})$ and $u^{r_1} \otimes u^{r_2}= u^{r_1+r_2}$, we obtain
$$\int_G \Phi(K+g^{-1} L, u^{r_1} \otimes (gu)^{r_2} ) \; dg= \left[ a^{r_1,r_2}\circ   M^{r_1+r_2}\right]	(\Phi)(K,L).$$
\end{proof}

\subsection{Moment maps for unitary area measures} We consider now the case $V=\CC^n=\RR^{2n}$ with $G=U(n)$. 
 When restricted to the space of unitary area measures, both $M^0$ and $M^1$ have non-trivial kernels. The kernel of the latter map was determined by the author 
in \cite{wannerer13}. 
The second order moment map, however, turns out to be injective.

As in \cite{wannerer13} we denote by $(z_1,\ldots, z_n, \zeta_1,\ldots, \zeta_n)$ the canonical coordinates on $\CC^n \oplus \CC^n \cong T\CC^n$, $z_i= x_i + \ii y_i$ and $\zeta_i= \xi_i + \ii \eta_i$. 
The canonical complex structure on $\CC^n$, i.e.\ componentwise multiplication by $\ii $, is denoted by $J\colon \CC^n \to \CC^n$.  Moreover we denote the action of $J$ on $T\CC^n\cong \CC^n \oplus \CC^n$ by the same
letter. We let $e_i$ be the element of $\CC^n$ with coordinates $z_j=\delta_{ij}$, $i,j=1,\ldots,n$ and put $e_{\bar i}:= J e_i$.

\begin{theorem}
 The map $M^2\colon \Area^{U(n)}\rightarrow \Val^{2, U(n)}$ is injective.
\end{theorem}

\begin{proof}
 Suppose the unitary area measure $\Phi$ satisfies  $M^2(\Phi)=0$. Denoting by
$Q=\sum_{i=1}^n \left(e_i^2 + e_{\bar i}^2\right)$ the metric tensor, we have
$$0 = (M^2(\Phi)(K),Q ) = \glob(\Phi)(K)$$
for every convex body $K\subset \CC^n$.
Hence, by \eqref{eq:glob_DN}, there exist numbers $c_{k,q}$ such that
$$\Phi= \sum_{k,q} c_{k,q} N_{k,q}.$$ 
As in \cite{bernig_fu11}, we denote by $E_{k,q}= \CC^q\oplus \RR^{k-2q}$ a two-parameter family of distinguished real subspaces of $\CC^n$. If $K\subset E_{k,q}$ is a convex body, then
$$M^2(\Phi)(K)= c_{k,q} M^2(N_{k,q})(K).$$
In particular, if $K$ equals the unit ball in $E_{k,q}$,  $K=B(E_{k,q})$, then
$$(M^2(\Phi)(K), e_n^2)= c_{k,q} \frac{2(n-k+q)}{2n-k} \int_{B(E_{k,q})\times S(E_{k,q}^\perp)} \xi_n^2 (\gamma_{k,q}-\beta_{k,q}).$$
The theorem will be proved if we can show  that the above integral is not zero. Denoting by $\iota \colon E_{k,q}\oplus E_{k,q}^\perp\to \CC^n\oplus \CC^n$ the inclusion map, we obtain
$$\iota^* \theta_0 =\sum_{i=k-q+1}^n d\xi_i\wedge d\eta_i,\qquad    \iota^* \theta_1 =\sum_{i=q+1}^{k-q} dx_i\wedge d\eta_i,\qquad    \iota^* \theta_2 =\sum_{i=1}^{q} dx_i\wedge dy_i,$$
and 
$$ \iota^* \beta =-\sum_{i=q+1}^{k-q} \eta_i dx_i,\qquad \text{and}\qquad     \iota^* \gamma =\sum_{i=k-q+1}^{n} \xi_i d\eta_i -\eta_i d\xi_i.$$
Consequently,
$$\iota^* \beta_{k,q} = \frac{1}{(k-2q)\omega_{2n-k}} \sum_{i=q+1}^{k-q} \eta_i \partial \eta_i  \lrcorner  \vol_{E_{k,q}\oplus E_{k,q}^\perp}$$
and
$$\iota^* \gamma_{k,q} = \frac{1}{2(n-k+q)\omega_{2n-k}} \sum_{i=k-q+1}^{n} \left(\xi_i \partial \xi_i +  \eta_i \partial \eta_i \right) \lrcorner  \vol_{E_{k,q}\oplus E_{k,q}^\perp}.$$

For every bounded Borel function $f\colon S^{n-1}\to \RR$ we have the identity
\begin{equation}
 \label{eq:contraction_formula}
\int_{S^{n-1}} f(x) x_i \; \partial x_i\lrcorner \vol_{\RR^n} = \int_{S^{n-1}} f(x) x_i^2 \; d\calH^{n-1}(x),\qquad i=1,\ldots, n.
\end{equation}
Here $\calH^{n-1}$ denotes the $(n-1)$-dimensional Hausdorff measure. Moreover, a calculation shows that 
\begin{equation}
 \label{eq:moments}
\int_{S^{n-1}} x_i^2 x_j^2 \; d\calH^{n-1}(x)= \left\{ \begin{array}{ll}
                                                          \frac{\omega_n}{n+2}, &  i\neq j;\\
							  \frac{3\omega_n}{n+2}, &  i=j,
                                                        \end{array}
						\right.
\end{equation}
see \cite{folland99}, Chapter~2, Exercise~63.
Using \eqref{eq:contraction_formula} and \eqref{eq:moments}, we obtain
$$\int_{B(E_{k,q})\times S(E_{k,q}^\perp)} \xi_n^2 (\gamma_{k,q}-\beta_{k,q})= \frac{\omega_k}{(n-k+q)(2n-k+2)}$$
which finishes the proof.

\end{proof}

\section{Proof of \texorpdfstring{$\bar v^2=0$}{v2=0}} \label{sec:vbar2}

We define three differential forms with values in $\Sym \CC^n$ which are invariant under the $U(n)$-action \eqref{eq:G-action}. We define the $0$-form 
$$w= \sum_{i=1}^n e_i \xi_i + e_{\bar i} \eta_i$$
and the $1$-forms
$$\nu_0= \sum_{i=1}^n e_i d\xi_i + e_{\bar i} d\eta_i \qquad \text{and}\qquad \nu_1= \sum_{i=1}^n e_i dx_i + e_{\bar i} dy_i.$$
Observe that $dw = \nu_0$, $$J ^*w= \sum_{i=1}^n e_{\bar i} \xi_i - e_i \eta_i,\qquad J^*\nu_0= \sum_{i=1}^n e_{\bar i} d\xi_i - e_i d\eta_i,\qquad J^*\nu_1= \sum_{i=1}^n e_{\bar i} dx_i - e_i dy_i,$$
and that also $J^*w$, $J^*\nu_0$, and  $J^*\nu_1$ are $U(n)$-invariant.

Using these forms, we define a tensor  valuation $\Phi_1\in \Val^{2,U(n)}$ by
$$\Phi_1(K) = \int_{N(K)} J^*w \; \nu_1 \wedge \theta_2^{n-1}.$$
Note that $\Phi_1$ is homogeneous of degree $2n-1$.

\begin{proposition}\label{lem:sym4}
If $n\geq 2$, then
$$(M^4(B_{2,0}), \Phi_1 * \Phi_1)=  (M^4(\Gamma_{2,1}), \Phi_1 * \Phi_1)= 0.$$
If $n \geq 3$, then also
$$(M^4(\Gamma_{2,0}), \Phi_1 * \Phi_1)= 0.$$
\end{proposition}
\begin{proof}All we have to do is to plug the Rumin differentials computed in the Appendix into \eqref{eq:poincare} and \eqref{eq:convproduct}. 
This computation, which at first sight looks a bit lengthy,
is simplified by the following two facts. 

If $\omega\in \Omega^{2n-1}(S\CC^n, \Sym^4\CC^n)$ and $\omega_1,\omega_2\in \Omega^{2n-1}(S\CC^n, \Sym^2\CC^n)$ are invariant with respect to the $U(n)$-action \eqref{eq:G-action}, then the form 
\begin{equation}\label{eq:poincare_form}( *_1^{-1}( *_1 \omega_1 \wedge *_1D\omega_2),D\omega)\end{equation}
is $U(n)$-invariant and as such already determined by its value at a single point $p\in S\CC^n$. In the following, we choose $p=(0,e_1)$. 

Moreover, the assertion that
$$\text{adding multiples of }\alpha \text{ to } \omega_1 \text{ does not change } \eqref{eq:poincare_form}$$
follows immediately from the fact that $D\omega$ is a multiple of $\alpha$.

For the rest of the proof we put $\omega_1=\omega_2=J^*w \; \nu_1 \wedge \theta_2^{n-2}$. The form $\omega$ will be either 
$$  w^4 \beta_{2,0}, \qquad w^4 \gamma_{2,0},\qquad \text{or}\qquad w^4\gamma_{2,1}.$$
At the point $p$  we have $\alpha=dx_1$ and 
\begin{equation}\label{eq:staromega}
 *_1\omega_1 \equiv  -(n-1)! e_{\bar 1}^2 dx_1
\end{equation}
up to forms which are not multiples $dx_1$.
To pick out only the relevant terms of $*_1D\omega_1$, observe that $ D\omega$ is a multiple of $ e_1^2$ at the point $p$. Since $ e_1^2$ does not appear
 in \eqref{eq:staromega}, only those terms in $*_1D\omega_1$ contribute to \eqref{eq:poincare_form} 
which are multiples of $ e_1^2$. Thus,
\begin{equation}\label{eq:starDomega_1} *_1 D\omega_1 \equiv 2 (n-1)! e_1^2 dy_1 d\eta_1.\end{equation}
 It follows from \eqref{eq:staromega} and \eqref{eq:starDomega_1} that the relevant part of $*_1^{-1}(*_1\omega_1 \wedge *_1D\omega_1)$ equals 
$$2(n-1)!^2 e_1^2 e_{\bar 1}^2 dx_2dy_2 \cdots dx_ndy_n d\eta_1.$$
Since every term in $D\omega$ which is a multiple of $ e_1^2e_{\bar 1}^2$ is a multiple of $d\eta_1$ as well, this immediatly implies that
\eqref{eq:poincare_form} vanishes. 
\end{proof}

\begin{lemma}
 If $n\geq 2$, then
\begin{equation}\label{eq:poincare_G10}(M^2(\Gamma_{1,0}),\Phi_1) = 0 \end{equation}
and 
\begin{equation}\label{eq:poincare_B10}
 (M^2(B_{1,0}),\Phi_1) = \frac{4 n! \omega_{2n}}{\omega_{2n-1}}.
\end{equation}
\end{lemma}

\begin{proof}
Again we are going to use  Lemma~\ref{lem:poincare} and the Rumin differentials computed in the Appendix. By $U(n)$-invariance it will be sufficient to compute 
$(\omega , D\omega')$ at the point $p=(0,e_1)$. If $\omega=w^2 \beta_{1,0}$ or $\omega=w^2 \gamma_{1,0}$, then at the point $p$ the form $\omega$ is a multiple of $ e_1^2$. Hence only the terms of $D\omega'$
which are multiples of $ e_1^2$ are relevant. At $p$ we have
$$D(Jw \; \nu \wedge \theta_2^{n-1}) \equiv 2 (n-1)! e_1^2 dx_1 dx_2 dy_2 \cdots dx_n dy_n d\eta_1$$
up to terms which are not multiples of $ e_1^2$. 
From this we immediately obtain \eqref{eq:poincare_G10} and \eqref{eq:poincare_B10}.
\end{proof}

\begin{theorem}\label{thm:vbar} $\bar v^2=0$.
\end{theorem}
\begin{proof}
 If $\Psi\in \{B_{2,0},\Gamma_{2,0}, \Gamma_{2,1}\}$, then there exist constants $c_{1}^\Psi$, $c_{2}^\Psi$, and, $c_{3}^\Psi$ such that 
$$A(\Psi)= c_{1}^\Psi B_{1,0}\otimes B_{1,0} + c_{2}^\Psi (B_{1,0}\otimes \Gamma_{1,0} + \Gamma_{1,0}\otimes B_{1,0}) + c_{3}^\Psi \Gamma_{1,0}\otimes \Gamma_{1,0}.$$ 
By the definition of the kinematic product, $\bar v^2=0$ is equivalent to $c^\Psi_{1}=0$.
The ftaig for tensor valuations (Theorem~\ref{thm:ftaig}) and Proposition~\ref{lem:sym4} imply that
$$\left\langle a^{2,2}\circ M^4(\Psi), \widehat{\PD}(\Phi_1)\otimes \widehat{\PD}(\Phi_1)\right\rangle= \left\langle \Phi_1 * \Phi_1, \widehat{\PD}(M^4(\Psi))\right\rangle = 0.$$
On the other hand, by Proposition~\ref{prop:tensorkin} and \eqref{eq:poincare_G10}, we have
$$ \left\langle a^{2,2}\circ M^4(\Psi), \widehat{\PD}(\Phi_1)\otimes \widehat{\PD}(\Phi_1)\right\rangle =\left\langle M^2\otimes M^2 \circ A(\Psi), \widehat{\PD}(\Phi_1)\otimes \widehat{\PD}(\Phi_1)\right\rangle = c_1^\Psi (M^2(B_{1,0}),\Phi_1)^2.$$
From \eqref{eq:poincare_B10} we conclude that $c_1^\Psi=0$. 
\end{proof}

\begin{proof}[Conclusion of the proof of Theorem~A]
 Let $h$ be the unique algebra homomorphism 
$$h\colon \RR[s,t,v]\to \Area^{U(n)*}$$
 determined by $t\mapsto \bar t$, $s\mapsto \bar s$, $v \mapsto \bar v$. It follows from Corollary~\ref{cor:subalgebra}, Lemma~\ref{lem:relations} and Theorem~\ref{thm:vbar}
that $h$ descends to an algebra homomorphism $\bar h \colon \RR[s,t,v]/I_n\rightarrow \Area^{U(n)*}$. By Proposition~\ref{prop:generators}, $\bar h$ is surjective. The theorem will be proved if we can show that
\begin{equation}\label{eq:dimension} \dim  \RR[s,t,v]/I_n = \dim \Area^{U(n)*}.\end{equation}

To this end let $M$ be the module given by the action of $\RR[s,t]$ on $\RR[s,t,v]$, let $S_n$ be the submodule generated by 
$$f_{n+1}(s,t), \quad f_{n+2}(s,t),\quad f_{n+1}(s,t)v, \quad f_{n+2}(s,t)v,\quad p_{n}(s,t)-q_{n-1}(s,t)v, \quad \text{and}\quad p_n(s,t)v,$$
 and let $T$ be the 
submodule generated by $\{v^k\colon k\geq 2\}$. Then $I_n=S_n + T$, $S_n\cap T=\{0\}$, and 
$$ M/I_n \cong (M/T ) / ((S_n+T)/T) \cong (M/T) / (S_n/(S_n\cap T)) =(M/T)/S_n.$$
Since 
$$M/T\cong \RR[s,t]\oplus \RR[s,t],$$
  \eqref{eq:dimension} follows now from Theorem~4.3 of \cite{wannerer13}.

\end{proof}

\section{Explicit local kinematic formulas} \label{sec:explicit}

In this final section we want to demonstrate how our main theorem can be used to derive explicit kinematic formulas. 

\begin{lemma}\label{lem:BB0}
 $$B_{k,q}^* * B^*_{k',q'}=0.$$
\end{lemma}
\begin{proof}
 This follows from Lemma~\ref{lem:Bideal} and $\bar v^2=0$. 
\end{proof}

An immediate consequence of Lemma~\ref{lem:BB0} and \eqref{eq:B*equals} is  
\begin{align}  N_{k,q}^* * N_{k',p'}^* = 
  \frac{(k-2q)(k'-2q')}{4(n-k+q)(n-k'+q')} \bigg( & \frac{2(n-k'+q')}{k'-2q'} \Delta_{k,q}^* * N_{k',q'}^* \label{eq:NN}  \\
 &+\frac{2(n-k+q)}{k-2q} \Delta_{k',q'}^* * N_{k,q}^*-\Delta_{k,q}^* * \Delta_{k',q'}^*\bigg).  \notag
\end{align}
This is the final ingredient we needed to obtain explicit local kinematic formulas. Indeed, Proposition~\ref{prop:Delta_st} gives us an explicit expression for $\Delta_{k,q}^*$ 
in terms of $\bar s$ and $\bar t$.
Hence, using Lemma~\ref{lem:sbartbar} repeatedly, we can compute $\Delta_{k,q}^* * \Delta_{k',q'}^*$ and $\Delta_{k,q}^* * N_{k',q'}^*$. Using relation \eqref{eq:NN} we can also compute $N_{k,q}^* * N_{k',p'}^*$.

Let us demonstrate this general procedure in a few examples.

\begin{example}
In the complex plane --- the simplest non-trivial case --- it is not difficult to write down the full array of kinematic formulas. We have
\begin{align*} A(\Delta_{3,1})  =  &(\Delta_{0,0}\otimes \Delta_{3,1}+\Delta_{3,1}\otimes \Delta_{0,0}) + \frac{2}{3} \left(\Delta_{1,0}\otimes \Delta_{2,0} +\Delta_{2,0}\otimes \Delta_{1,0} \right)\\
                                    &+ \frac{1}{3}\left(  \Delta_{1,0}\otimes \Delta_{2,1} +\Delta_{2,1}\otimes \Delta_{1,0}\right)+ \frac{1}{3}\left( N_{1,0} \otimes \Delta_{2,0} +  \Delta_{2,0} \otimes N_{1,0} \right) \\
       &  -\frac{2}{3}\left( N_{1,0} \otimes \Delta_{2,1} +  \Delta_{2,1} \otimes N_{1,0} \right),\\
 A(\Delta_{2,1})  =  &\left(\Delta_{0,0}\otimes \Delta_{2,1}+\Delta_{2,1}\otimes \Delta_{0,0}\right) + \frac{4}{9\pi }\left(  \Delta_{1,0}\otimes N_{1,0} +N_{1,0}\otimes \Delta_{1,0}\right)\\
                                    & + \frac{8}{9\pi} \Delta_{1,0}\otimes \Delta_{1,0} + \frac{2}{9\pi }N_{1,0} \otimes N_{1,0},  \\
 A(\Delta_{2,0})  =  &\left(\Delta_{0,0}\otimes \Delta_{2,1}+\Delta_{2,1}\otimes \Delta_{0,0}\right) -\frac{4}{9\pi }\left(  \Delta_{1,0}\otimes N_{1,0} +N_{1,0}\otimes \Delta_{1,0}\right)\\
                                    & + \frac{16}{9\pi} \Delta_{1,0}\otimes \Delta_{1,0} -\frac{8}{9\pi }N_{1,0} \otimes N_{1,0} 
\end{align*}
and the trivial formulas 
\begin{align*}A(\Delta_{1,0}) & =  \Delta_{0,0}\otimes \Delta_{1,0}+\Delta_{1,0}\otimes \Delta_{0,0}, \\
	      A(N_{1,0}) & =  \Delta_{0,0}\otimes N_{1,0}+N_{1,0}\otimes \Delta_{0,0}, \\
	      A(\Delta_{0,0})  & =  \Delta_{0,0}\otimes \Delta_{0,0}.
\end{align*}  
 Let us show how the coefficients of $N_{1,0} \otimes N_{1,0}$ in  $A(\Delta_{2,1})$ and $A(\Delta_{2,0})$ are obtained. By Proposition~\ref{prop:Delta_st}, 
$$\Delta_{1,0}^* = \frac{2}{3} \bar t.$$
Using Lemma~\ref{lem:sbartbar}, we find that
$$\Delta_{1,0}^* * \Delta_{1,0}^* = \frac{2}{3}\bar t* \Delta_{1,0}^*= \frac{8}{9\pi}\left(2\Delta_{2,0}^* + \Delta_{2,1}^*\right)$$
and
$$\Delta_{1,0}^* * N_{1,0}^*=  \frac{2}{3}\bar t * N_{1,0}^* = \frac{4}{9\pi}\left(-\Delta_{2,0}^* + \Delta_{2,1}^*\right).$$
From this we can read of the coefficients of $\Delta_{1,0}\otimes \Delta_{1,0}$ and $\Delta_{1,0}\otimes N_{1,0}$ in $A(\Delta_{2,1})$ and $A(\Delta_{2,0})$. 
Using \eqref{eq:NN}, we obtain 
$$N_{1,0}^* * N_{1,0}^* = \Delta_{1,0}^* * N_{1,0}^* -\frac{1}{4} \Delta_{1,0}^* * \Delta_{1,0}^* = \frac{2}{9\pi} \left( \Delta_{2,1}^*  - 4\Delta_{2,0}^*\right)$$
which gives us the coefficients of $N_{1,0} \otimes N_{1,0}$.

\end{example}

\begin{example}
 In $\CC^3$, let us determine the coefficient of $N_{2,0}\otimes N_{3,1}$ in $A(\Delta_{5,2})$. By equation \eqref{eq:NN},
$$N_{2,0}^* * N_{3,1}^* = \frac{1}{2} \left(  2 \Delta_{2,0}^* * N_{3,1}^* + \Delta_{3,1}^* * N_{2,0}^* - \Delta_{2,0}^* * \Delta_{3,1}^*\right).$$ 
Using Proposition~\ref{prop:Delta_st} and Lemma~\ref{lem:sbartbar}, we find that
$$\Delta_{3,1}^* * N_{2,0}^* =\frac{2\pi}{9}\bar t \left(\bar s * N_{2,0}^* \right)= \frac{\bar t}{18}\left( \Delta_{4,2}^* -6 \Delta_{4,1}^*\right)= -\frac{5}{18}\Delta_{5,2}^*$$
and
$$ \Delta_{3,1}^* * \Delta_{2,0}^* =  \frac{2\pi}{9}\bar t \left(\bar s * \Delta_{2,0}^* \right) = \frac{\bar t}{18}\left( \Delta_{4,2}^* +6 \Delta_{4,1}^*\right)=\frac{7}{18} \Delta_{5,2}^*. $$
Similarly, we obtain
$$\Delta_{2,0}^* * N_{3,1}^*  = \left( \frac{\pi}{8} \bar t^2 - \frac{\pi}{12} \bar s\right) * N_{3,1}^*=  \frac{1}{9} \Delta_{5,2}^*$$
and hence
$$N_{2,0}^* * N_{3,1}^* = \frac{1}{2}\left( \frac{2}{9}-\frac{5}{18}-\frac{7}{18}\right) \Delta_{5,2}^*=  -\frac{2}{9} \Delta_{5,2}^*.$$
We conclude that the coefficient of $N_{2,0}\otimes N_{3,1}$ in $A(\Delta_{5,2})$ equals $-\frac{2}{9}$.
\end{example}

As the dimension $n$ increases, computing explicit kinematic formulas by hand does not become more difficult, but increasingly tedious. 
Since the procedure itself is very simple, it can be implemented in any computer algebra package.

\section*{Appendix}

The purpose of this appendix is to provide the explicit expressions of certain Rumin differentials we needed in Section~\ref{sec:vbar2}. 

Recall that if $M$ is a contact manifold of dimension $2n-1$ with global contact form $\alpha$, then there exists a canonical 
second order differential operator $D\colon \Omega^{n-1}(M)\to \Omega^{n}(M)$ called Rumin differential. It is defined as follows:
$$D\omega = d(\omega + \alpha \wedge \xi ),$$
where $\xi\in \Omega^{n-2}(M)$ is chosen such that $d(\omega + \alpha \wedge \xi )$ is a multiple of the contact form. Rumin \cite{rumin94} showed
 that such a form $\xi$ always exists and that moreover $\alpha\wedge \xi$ is unique. Recall also that the Reeb vector field $T$ is the unique smooth vector field on $M$ such that
$$T\lrcorner \alpha = 1 \qquad \text{and}\qquad T\lrcorner d\alpha =0.$$ 
Using the Reeb vector field we may reformulate Rumin's theorem as follows: For every $(n-1)$-form $\omega$ on $M$ there exists unique form $\xi=:Q(\omega)$ with $T\lrcorner \xi=0$ such that $d(\omega + \alpha \wedge \xi)$ is 
a multiple of the contact form.

On the sphere bundle $S\CC^n$ the standard contact form is given by 
$$\alpha=\sum_{i=1}^n \xi_i dx_i + \eta_i dy_i,$$
 which agrees with the notation we have used so far. We remark that $f^* Q(\omega)=Q(f^*\omega)$ for every smooth map $f\colon M \to M$ satisfying $f^* \alpha =\alpha$. 
Note that, in particular, every isometry of $\CC^n$ lifted to the sphere bundle has this property. This is very useful for us: Since we compute Rumin differentials of vector-valued forms $\omega$ invariant under the 
$U(n)$-action \eqref{eq:G-action}, 
it will be sufficient to compute $Q(\omega)$ at a single point $p$. In the following we choose $p=(0,e_1)$.

\begin{proposition*}
\begin{align} 
 D(Jw \; \nu_1 \wedge \theta_2^{n-1}) =\ & \alpha \wedge \Big( 2(n-1) J^*\nu_0 \wedge \nu_1 \wedge \beta - J^*w \; \nu_0 \wedge \theta_2 - 2 w \; J^* \nu_0 \wedge \theta_2 \label{eq:D2}\\
					& -(n-1) J^*w \; \nu_1\wedge  \theta_1\Big) \wedge \theta_2^{n-2},\notag\\
 D(w^4 \beta \wedge \theta_1 \wedge \theta_0^{n-2}) =\ & 4w^3 \alpha \wedge \Big( 7\nu_0\wedge\beta \wedge \gamma + 2 w \gamma\wedge \theta_1 + J^*w \gamma\wedge \theta_s - \nu_0\wedge \theta_s   \label{eq:D3} \\
					      & + 2 J^*\nu_0 \wedge \theta_1 - 3w \beta \wedge \theta_0\Big) \wedge \theta_0^{n-2}\notag\\
					& +12 w^2 \alpha \wedge \Big( J^*w \nu_0\wedge \theta_1 -\nu_0 \wedge J^*\nu_0\wedge \beta\Big) \wedge \theta_0^{n-2},\notag\\
 D(w^4 \gamma \wedge \theta_2\wedge \theta_0^{n-2}) =\ & 2w^4 \alpha \wedge \Big( (2n+3) \beta \wedge \theta_0 - (n+1) \gamma \wedge \theta_1 \Big) \wedge \theta_0^{n-2}	\label{eq:D4}\\
						 & + 4w^3 \alpha \wedge \Big( \nu_0 \wedge \theta_s + 2 J^*\nu_1 \wedge \theta_0 - (2n+3) \nu_0 \wedge \beta\wedge \gamma - J^*w \gamma\wedge \theta_s\Big) \wedge \theta_0^{n-2}\notag\\
						 & -12 w^2 \alpha \wedge \nu_0\wedge J^* \nu_1 \wedge \gamma \wedge \theta_0^{n-2},\notag
\intertext{for $n\geq 2$. If $n\geq 3$, then}
 D(w^4 \gamma\wedge \theta_1^2 \wedge \theta_0^{n-3}) =\ & 8w^4  \alpha \wedge \Big( \gamma \wedge \theta_1 - \beta \wedge \theta_0 \Big) \wedge \theta_0^{n-2} 		\label{eq:D5}\\
							& +4 w^3 \alpha \wedge \Big( 4 \nu_0 \wedge \beta \wedge \gamma  + 4 J^*\nu_0 \wedge \theta_1 \Big) \wedge \theta_0^{n-2}\notag\\
							& -24 w^2 \alpha \wedge \nu_0 \wedge J^*\nu_0 \wedge \gamma \wedge \theta_1 \wedge \theta_0^{n-3}.\notag
\end{align}

\end{proposition*}

Note that a the point $(0,e_1)\in S\CC^n$ we have
$$\alpha= dx_1, \qquad \beta=dy_1,\qquad \gamma=d\eta_1, \qquad \text{and}\qquad d\xi_1=0.$$
Recall also that 
$$ T\lrcorner \alpha = 1, \qquad T \lrcorner \beta =0, \qquad T \lrcorner \gamma =0,$$
and 
$$  T\lrcorner \theta_0 = 0, \qquad  T\lrcorner \theta_1 = \gamma, \qquad  T\lrcorner \theta_2 = \beta.$$
Moreover,
$$ T\lrcorner \nu_0 =  T\lrcorner J^*\nu_0 =0, \qquad T\lrcorner \nu_1 =w, \qquad \text{and}\qquad \qquad T\lrcorner J^*\nu_1 = J^*w.$$

\begin{lemma*}

\begin{align}
Q(Jw \; \nu_1 \wedge \theta_2^{n-1}) =\ & J^*w(w \theta_2 -(n-1) v_1 \wedge \beta) \wedge \theta_2^{n-2},\label{eq:Q2} \\
\intertext{and up to multiples of the contact form}
 Q(w^4 \beta \wedge \theta_1 \wedge \theta_0^{n-2}) \equiv\ & w^3 \big( w \theta_s - 4 J^*w \theta_1 - 4 \beta \wedge J^*\nu_0  - 6 w \beta \wedge \gamma  \big) \wedge \theta_0^{n-2} ,\label{eq:Q3}\\
 Q(w^4 \gamma \wedge \theta_2\wedge \theta_0^{n-2}) \equiv\ &  w^3 \big( 2(n+1)w \beta \wedge \gamma \wedge \theta_0 + 2(n-2) \gamma \wedge \nu_0\wedge \theta_s - (n-1) w \theta_0 \wedge \theta_s \label{eq:Q4} \\
& - 4 \gamma \wedge J^*\nu_1 \wedge \theta_0  \big)\wedge \theta_0^{n-3},\notag\\
Q(w^4 \gamma\wedge \theta_1^2 \wedge \theta_0^{n-3}) \equiv \ & 2 w^3 \big( w \theta_0 \wedge \theta_s -2w \beta \wedge \gamma \wedge \theta_0 -2 \gamma \wedge \nu_0 \wedge \theta_s  -4 \gamma \wedge J^*\nu_0 \wedge \theta_1
\big) \wedge \theta_0^{n-3}. \label{eq:Q5}
\end{align}

\end{lemma*}
\begin{proof} We write $Q(\omega)$ for the left-hand side of \eqref{eq:Q2} and $\xi$ for the right-hand side. We first show 
\begin{equation}
\label{eq:vertical} 
 \alpha \wedge d\omega + \alpha \wedge d \alpha \wedge \xi =0.
\end{equation}
Note that at $p=(0,e_1)$ we have 
\begin{equation}
 \label{eq:alpha_theta2}
\alpha \wedge \theta_2^{n-2}= -(n-2)! \sum_{i=2}^n \partial y_1 \wedge \partial x_i \wedge \partial y_i \lrcorner \vol_{\CC^n}
\end{equation}
and hence
\begin{align*}
 -(n-1) J^*w \nu_1 \wedge \beta \wedge d\alpha \wedge \alpha \wedge \theta_2^{n-2} & = (n-1)! J^*w \nu_1 \wedge d\alpha \wedge \left(\sum_{i=2}^n  \partial x_i \wedge \partial y_i \lrcorner \vol_{\CC^n} \right)\\
 & = (n-1)! e_{\bar 1}  \left(\sum_{i=2}^n e_{\bar i} d\xi_i - e_{ i} d\eta_i\right) \wedge \vol_{\CC^n}.
\end{align*}
Moreover, we have at the point $p$ 
$$w J^* w \alpha \wedge d\alpha \wedge  \theta_2^{n-1} = -(n-1)! e_1e_{\bar 1}  d\eta_1 \wedge \vol_{\CC^n}$$
and 
$$\alpha \wedge d\omega = -(n-1)! e_{\bar 1} J^* \nu_0 \wedge \vol_{\CC^n}.$$
This proves \eqref{eq:vertical} and, since $T\lrcorner \xi=0$, we conclude that $Q(\omega)=\xi$ at $p$ and the $U(n)$-invariance implies now \eqref{eq:Q2}.

To prove \eqref{eq:Q3} first note that at $p$ we have
$$\theta_0^{n-2}= (n-2)! \sum_{i=2}^n \partial  \eta_1 \wedge \partial \xi_i \wedge \partial \eta_i \lrcorner \vol_{T_pS\CC^n},$$
where 
$$\vol_{T_pS\CC^n} = d\eta_1\wedge d\xi_2 \wedge d\eta_2 \wedge \cdots \wedge d\xi_n \wedge d\eta_n.$$
Put $\omega'= w^4 \beta \wedge \theta_1$ and 
$$\xi'=w^3 \big( w \theta_s - 4 J^*w \theta_1 - 4 \beta \wedge J^*\nu_0  - 6 w \beta \wedge \gamma  \big).$$
For  $i=2,\ldots,n$  we calculate
\begin{align*}
 \alpha \wedge d\omega' \wedge \left( \partial \eta_1 \wedge \partial \xi_i \wedge \partial \eta_i \lrcorner \vol_{T_pS\CC^n} \right) & =  \alpha \wedge \Big(4e_1^3 e_{\bar 1} d\eta_1\wedge dy_1 \wedge (dx_i \wedge d\eta_i -dy_i\wedge d\xi_i)\\
& \qquad \qquad - 4e_1^3 (e_i dx_i + e_{\bar i} dy_i)\wedge dy_1 \wedge d\xi_i \wedge d\eta_i \\ 
& \qquad \qquad -2e_1^4 dx_i\wedge dy_i\wedge d\xi_i \wedge d\eta_i \Big) \wedge \left( \partial \xi_i \wedge \partial \eta_i \lrcorner \vol_{T_pS\CC^n} \right)\\
& =-\alpha \wedge d\alpha \wedge \xi' \wedge \left(\partial \eta_1 \wedge \partial \xi_i \wedge \partial \eta_i \lrcorner \vol_{T_pS\CC^n} \right).
\end{align*}
This establishes \eqref{eq:Q3} and \eqref{eq:Q4} is proved in the same way.

We come now to the proof of \eqref{eq:Q5}. We split $d\omega$ into two summands
$$d\omega= 2w^4 \theta_1^2\wedge \theta_0^{n-2}  + 4w^3 \nu_0 \wedge \gamma \wedge \theta_1^2 \wedge \theta_0^{n-3} =:\Omega_1 + \Omega_2$$
and similarly write $\xi= \Xi_1 + \Xi_2$ with
$$\Xi_1= 2 w^4 \big(\theta_s -2 \beta \wedge \gamma \big) \wedge \theta_0^{n-2}$$
and 
$$\Xi_2 =  4 w^3  \big(   \nu_0 \wedge \theta_s + 2   J^*\nu_0 \wedge \theta_1 \big) \wedge \gamma \wedge \theta_0^{n-3}.$$
As in the proof of \eqref{eq:Q3} and \eqref{eq:Q4}, one checks that $\alpha \wedge (\Omega_1 + d\alpha \wedge \Xi_1)=0$ at $p$. Proving  
$\alpha \wedge (\Omega_2 + d\alpha \wedge \Xi_2)=0$ at $p$ requires a bit more work. First observe that at the point $p$,
\begin{align*}\gamma \wedge \theta_0^{n-3}  &= (n-3)! \sum_{2\leq i<j\leq n}  \left( \partial \xi_i \wedge \partial \eta_i \wedge \partial \xi_j \wedge \partial \eta_j   \right) \lrcorner \vol_{T_pS\CC^n}\\
		& =: (n-3)! \sum_{2\leq i<j\leq n}  v_{ij}.
\end{align*}
We put 
$$\sigma_i = d\xi_i \wedge dx_i + d\eta_i \wedge dy_i \qquad \text{and} \qquad  \sigma_i' =  d\xi_i \wedge dy_i -d\eta_i \wedge dx_i,$$
for $i=1,\ldots, n$. 
Then $d\alpha = -\theta_s = \sum_i \sigma_i$ and $\theta_1= \sum_i \sigma_i'$. We compute at $p$
$$\alpha \wedge ( \theta_1^2 -(d\alpha)^2 ) \wedge v_{ij} =
 2 \alpha \wedge (\sigma_i' \wedge \sigma_j' - \sigma_i \wedge \sigma_j  )  \wedge v_{ij}, \qquad 2\leq i<j\leq n$$
 and
$$\alpha \wedge (d\alpha \wedge  \theta_1  ) \wedge v_{ij} =2 \alpha \wedge (\sigma_i' \wedge \sigma_j - \sigma_i \wedge \sigma_j'  )  \wedge v_{ij}, \qquad 2\leq i<j\leq n.$$
Hence
$$ \alpha \wedge (  \nu_0 \wedge( \theta_1^2 -(d\alpha)^2 )  + J^*\nu_0 \wedge d\alpha \wedge  \theta_1 ) \wedge v_{ij}= 0$$
and we conclude that also $\alpha \wedge (\Omega_2 + d\alpha \wedge \Xi_2)=0$ at $p$.
\end{proof}

\begin{proof}[Proof of the Proposition]
Since $d\omega+ d\alpha\wedge \xi $, $\xi\equiv Q(\omega)$, is a multiple of $\alpha$, we have
$$D\omega= \alpha \wedge \left( T\lrcorner (d\omega + d\alpha \wedge \xi) - d \xi\right).$$
If $\xi = Q(\omega)$, then 
$$D\omega= \alpha \wedge \left( T\lrcorner d\omega  - d \xi\right).$$
Using this, \eqref{eq:D2} follows immediately from
\begin{align*} T\lrcorner d\omega &=  (n-1) J^*\nu_0\wedge \nu_1 \wedge \beta \wedge \theta_2^{n-2} - w J^*\nu_0 \wedge \theta_2^{n-1}\\
\intertext{and}
d\xi &= \left( -(n-1) J^*\nu_0 \wedge \nu_1 \wedge  \beta +  (w J^*\nu_0  + J^* w \nu_0 )\wedge \theta_2 + (n-1) J^*w \nu_1 \wedge \theta_1  \right) \wedge \theta_2^{n-2}.
\end{align*}
In the same way \eqref{eq:D3}, \eqref{eq:D4}, and \eqref{eq:D5} can be proved.

\end{proof}

\end{document}